\documentclass[11pt]{article}
\usepackage{amsmath,amssymb,mathrsfs}
\usepackage{latexsym}
\usepackage[all]{xy}
\def\e#1\e{\begin{equation}#1\end{equation}}
\def\ea#1\ea{\begin{align}#1\end{align}}
\def\eq#1{{\rm(\ref{#1})}}
\newtheorem{thm}{Theorem}[section]
\newtheorem{lem}[thm]{Lemma}
\newtheorem{prop}[thm]{Proposition}

\newtheorem{cor}[thm]{Corollary}
\newenvironment{dfn}{\medskip\refstepcounter{thm}
\noindent{\bf Definition \thesection.\arabic{thm}\ }}{\medskip}
\newenvironment{ex}{\medskip\refstepcounter{thm}
\noindent{\bf Example \thesection.\arabic{thm}\ }}{\medskip}
\newenvironment{proof}[1][,]{\smallskip\ifcat,#1
\noindent{\it Proof.\ }\else\noindent{\it Proof of #1.\ }\fi}
{\relax\unskip\nobreak ~\hfill$\square$\bigskip}
\def\narrow{\par\addtolength{\leftskip}{20pt}\addtolength{\rightskip}{5pt}}
\def\oldmargins{\par\setlength{\leftskip}{0pt}\setlength{\rightskip}{0pt}}
\newcounter{quest}[section]
\newcounter{qa}[quest]
\newcounter{qi}[quest]

\def\inext{\par\ifnum\value{qi}=0 \narrow\fi \medskip\addtocounter{qi}{1}
\noindent\hbox to 0pt{\hss\bf(\roman{qi})\hskip .3em}}
\def\anext{\par\ifnum\value{qa}=0 \narrow\fi \medskip\addtocounter{qa}{1}
\noindent\hbox to 0pt{\hss\bf(\alph{qa})\hskip .3em}}

\numberwithin{equation}{section}

\newcommand{\be}{\begin{equation}}
\newcommand{\ee}{\end{equation}}
\newcommand{\bes}{\begin{equation*}}
\newcommand{\ees}{\end{equation*}}
\newcommand{\bea}{\begin{eqnarray}}
\newcommand{\eea}{\end{eqnarray}}

\def\w{\wedge}

\def\pa{\partial}
\def\bpa{{\bar \partial}}
\def\pas{\partial^*}
\def\bpas{{\bar \partial^*}}

\def\dpdl{d+\dl}
\def\dadl{d\,\cap\, \dl}
\def\dcap{d\, \cap\, \dl }

\def\hB{{\widehat B}}

\def\ds{d^*}
\def\dl{d^\Lambda}
\def\dls{d^\Lambda{}^*}
\def\dc{d^c}
\def\dcs{d^c{}^*}

\def\ss{*_s}

\def\CA{\Omega}
\def\CB{\mathcal{P}}
\def\CC{{\tilde \Omega}}
\def\tCB{{\tilde{\mathcal{P}}}} 
\def\CH{\mathcal{H}}
\def\CJ{\mathcal{J}}
\def\CL{\mathcal{L}}

\def\mC{\mathcal{C}}

\def\La{\Lambda}

\def\a{\alpha}
\def\b{\beta}

\def\la{\lambda}
\def\om{\omega}
\def\th{\theta}
\def\vp{\varphi}
\def\bth{{\bar \theta}}
\def\rx{{\rho_X}}

\def\im{{\rm im~}}
 
\def\ea{e_1}
\def\eb{e_2}
\def\ec{e_3}
\def\ed{e_4}
\def\xa{x_1}
\def\xb{x_2}
\def\xc{x_3}
\def\xd{x_4}

\def\Ca{C^{(1)}_l}
\def\Cb{C^{(2)}_l}
\def\Cc{C^{(12)}_{l+1}}

\def\ra{\rho^{(1)}_{2n-l}}
\def\rb{\rho^{(2)}_{2n-l}}
\def\rc{\rho^{(12)}_{2n-l-1}}

\begin{document}

\title{\bf{Cohomology and Hodge Theory on \\
 Symplectic Manifolds: I}}

\author{Li-Sheng Tseng and Shing-Tung Yau \\
\\
}

\date{September 29, 2009}

\maketitle

\begin{abstract} We introduce new finite-dimensional cohomologies on symplectic manifolds.  Each exhibits Lefschetz decomposition and contains a unique harmonic representative within each class.   
Associated with each cohomology is a primitive cohomology defined purely on the space of primitive forms.  We identify the dual currents of lagrangians and more generally coisotropic submanifolds with elements of a primitive cohomology, which dualizes to a homology on coisotropic chains.

\

\end{abstract}



\tableofcontents

\setcounter{equation}{0}
\setcounter{footnote}{0}


\

\section{Introduction}

The importance of Hodge theory in Riemannian and complex geometry is without question.  But in the symplectic setting, although a notion of symplectic Hodge theory was discussed in the late 1940s by Ehresmann and Libermann \cite{EH,Libermann} and re-introduced by Brylinski \cite{Brylinski} about twenty years ago, its usefulness has been rather limited.  To write down a symplectic adjoint, one makes use of the symplectic star operator $\ss$, defined analogously to the Hodge star operator but with respect to a symplectic form $\om$ instead of a metric.  Specifically, on a symplectic manifold $(M,\om)$ with dimension $2n$, the symplectic star acts on a differential $k$-form by
\begin{align*}
A \w \ss A' &=  (\om^{-1})^k(A, A') \,d{\rm vol} \notag\\
& = \frac{1}{k!}(\om^{-1})^{i_1 j_1}(\om^{-1})^{i_2 j_2}\!\ldots (\om^{-1})^{i_k j_k}\, A_{i_1 i_2 \ldots i_k}\,A'_{j_1 j_2 \ldots j_k}\,\,  \frac{\om^n}{n!}
\end{align*}   
with repeated indices summed over.  The adjoint of the standard exterior derivative takes the form
$$\dl = (-1)^{k+1} \ss d\, \ss~,$$ 
acting on a $k$-form.
A differential form is then called ``symplectic harmonic" if it is both $d$-closed and $\dl$-closed.  As for the existence of such forms, Mathieu \cite{Mathieu} proved that every de Rham cohomology $H_d^*(M)$ class contains a symplectic harmonic form if and only if the symplectic manifold satisfies the strong Lefschetz property; that is, the map
$$H_d^{k}(M) \rightarrow H_d^{2n-k}(M)~,  ~~~A\to [\om]^{n-k} \w \,A~,$$
 is an isomorphism for all $k \leq n\,$. 

As it stands, the set of differential forms that are both $d$- and $\dl$-closed lacks certain intrinsic properties that we typically associate with harmonic forms.   Concerning existence, one would like that a symplectic harmonic form exists in every cohomology class for any symplectic manifold.  But Mathieu's theorem tells us that the existence of a symplectic harmonic form in all classes of the de Rham cohomology requires that the symplectic manifold satisfies the strong Lefschetz property.  Unfortunately, many known non-K\"ahler symplectic manifolds do not satisfy strong Lefschetz.   One would also like the uniqueness property of harmonic forms in each cohomology class to hold.   But consider for instance one-forms that are $d$-exact.  They are trivially $d$-closed, and it can be easily shown that they are always $\dl$-closed too.   Uniqueness of $d$- and $\dl$-closed forms within the de Rham cohomology class simply does not occur.  These two issues, of existence and uniqueness of $d$- and $\dl$-closed forms, indicate that perhaps the de Rham cohomology is not the appropriate cohomology to consider symplectic Hodge theory.  But if not de Rham cohomology, what other cohomologies are there on symplectic manifolds?  

In this paper, we introduce and analyze new cohomologies for compact symplectic manifolds.  In our search for new cohomologies, a simple approach is to start with the requirement of $d$- and $\dl$-closed and try to attain uniqueness by modding out some additional exact-type forms.  Having in mind the properties $(d)^2= (\dl)^2=0$ and the anti-commutivity $d\dl = - \dl d$, we are led to consider the following cohomology of smooth differential forms $\CA^*(M)$ on a symplectic manifold
\bes
H^k_{d+\dl}(M) = \frac{\ker (d + \dl) \cap \CA^k(M)}{{\rm im} \ d\dl \cap \CA^k(M)}\,,
\ees
noting that $\ker d \cap \ker \dl = \ker(d+ \dl)$ in $\CA^k(M)\,$.   Conceptually, in writing down such a cohomology, we have left behind the adjoint origin of $\dl$ and instead are treating $\dl$ as an independent differential operator.  Hence, our choice of notation, $\dl$, differs from the more commonly used $\delta$ symbol, denoting adjoint, in the literature.  By elliptic theory arguments, we shall show that for $M$ compact, $H^k_{d+\dl}(M)$ is indeed finite dimensional.  And since $H^k_{d+\dl}(M)$ is by construction invariant under any symplectmorphisms of a symplectic manifold, it is a good symplectic cohomology encoding global invariants.

As for the notion of harmonic forms, we will define it in the standard Riemannian fashion, utilizing the Hodge star operator, which requires a metric.  On any symplectic manifold $(M, \om)$, there always exists a compatible triple, $(\om, J, g)$, of symplectic form, almost complex structure, and Riemannian metric.  And it is with respect to such a compatible metric $g$ that we shall define the Hodge star operator $*$.  We will require the symplectic harmonic form for this cohomology to be not only $d$- and $\dl$-closed, but additionally also $(d\dl)^*=(-1)^{k+1}\!*(d\dl)\,*$ closed.  A unique harmonic form can then be shown to be present in every cohomology class of $H^k_{d+\dl}(M)$.

Appealingly, the cohomology $H^*_{d+\dl}(M)$ has a number of interesting properties.  As we will show, it commutes with the Lefschetz's decomposition of forms, and hence the Lefschetz property with respect to $H^*_{d+\dl}(M)$, instead of the de Rham cohomology $H_d^*(M)$, holds true on all symplectic manifolds.  It turns out that if the symplectic manifold satisfies the strong Lefschetz property with respect to $H^k_d(M)$, which is equivalent to the presence of the $d\dl$-lemma \cite{Merkulov, Guillemin}, then  $H^k_{d+\dl}(M)$ becomes isomorphic to de Rham $H_d^k(M)$.  Essentially, $H^*_{d+\dl}(M)$ contains the additional data of the symplectic form $\om$ (within the $\dl$ operator).  It is therefore a more refined cohomology then de Rham on a symplectic manifold, with possible dependence only on $\om$ and always Lefschetz decomposable.

From $H^*_{d+\dl}(M)$, we are led to consider other new finite-dimensional cohomologies.   For the de Rham cohomology, there is a natural pairing between the cohomologies $H_d^k(M)$ and $H_d^{2n-k}(M)$ via the wedge product.  For $H^*_{d+\dl}(M)$, there is also a natural pairing via the wedge product.  However, the pairing is not with itself but with the cohomology
\bes
H^k_{d\dl}(M) = \frac{\ker d\dl \cap \CA^k(M)}{({\rm im} \ d + {\rm im} \ \dl) \cap \CA^k(M)}~.
\ees
We will show that $H^{k}_{d+\dl}(M)$ and $H^{2n-k}_{d\dl}(M)$ forms a well-defined pairing that is non-degenerate.  As may be expected, we will also find that $H^*_{d\dl}(M)$ exhibits many of the same properties found in its paired cohomology $H^{*}_{d\dl}(M)$, including being Lefschetz decomposable.  These two cohomologies are indeed isomorphic to one another.


In a separate direction, we can use the fact that the three new cohomologies we have introduced are Lefschetz decomposable to consider their restriction to the subspace of smooth primitive forms, $\CB^*(M)$.  Such would be analogous on a K\"ahler manifold to the primitive Dolbeault cohomology.   As an example, the associated primitive cohomology of $H^*_{d+\dl}(M)$ can be written for $k\leq n$ as
\bes
PH^k_{\dpdl}(M) = \frac{\ker d \cap \CB^k(M)}{ d\dl \CB^k(M)}~,
\ees
which acts purely within the space of primitive forms $\CB^*(M)$.  These primitive cohomologies should be considered as more fundamental as for instance $PH^k_{\dpdl}(M)$ underlies $H^k_{d+\dl}(M)$ by Lefschetz decompositon, and indeed also $H^k_{d\dl}(M)$, by isomorphism with $H^{2n-k}_{d+\dl}(M)$.  

The appearance of primitive cohomologies is also interesting from a different perspective.  As we shall see, the currents of coisotropic submanifolds turns out to be exactly primitive.  That this is so is perhaps not that surprising as the coisotropic property is defined with reference to a symplectic form $\om$, just like the condition of being primitive.  Motivated by de Rham's theorem relating $H_d(M)$ cohomology with the homology of chains, the existence of primitive cohomologies is at once suggestive of the following homology
\bes
PH_l(M)= \frac{\ker \partial \cap  \mC_l(M)}{\partial\, \mC_{l+1}(M)}~,
\ees
where ${\mC}_l$ with $n \leq l < 2n\,$ consists of the subspace of $l$-chains that are coisotropic (and if with boundaries, the boundaries are also  coisotropic).  We will associate such a homology with a finite primitive cohomolgy that is a generalization of $PH_{\dpdl}(M)$.   A list of the new cohomologies we introduce in this paper can be found in Table \ref{tabcoh}.

\begin{table}[t]
{\renewcommand{\arraystretch}{2.2} 
\renewcommand{\tabcolsep}{0.3cm}
\begin{tabular}{l l |  l }
 & Cohomology (Primitive Cohomology)& \qquad\qquad\qquad Laplacian \\
\hline
$1.$&  $H^k_{d+\dl}(M) = \dfrac{\ker (d + \dl) \cap \CA^k(M)}{{\rm im} \ d\dl \cap \CA^k(M)}$  &  $\Delta_{d+\dl}= d\dl (d\dl)^* + \la( \ds d + \dls \dl) $\\
& $PH^k_{\dpdl}(M) = \dfrac{\ker d \cap \CB^k(M)}{d\dl \CB^k(M)}$ & $\Delta^p_{d+\dl}= d\dl (d\dl)^* + \la\, \ds d $ \\ 
\hline
$2.$&  $H^k_{d\dl}(M)= \dfrac{\ker d\dl \cap \CA^k}{(\im d + \im \dl)\cap \CA^k}$ & $\Delta_{d\dl}= (d\dl)^* d\dl + \la (d\ds + \dl\dls)$  \\
& $PH^k_{d\dl}(M) = \dfrac{\ker d\dl \cap \CB^k(M)}{ (d + L\, H^{-1} \dl) \CB^{k-1} + \dl \CB^{k+1}}$ & $\Delta^p_{d\dl}=(d\dl)^* d\dl + \la\, \dl \dls$ \\  
\end{tabular}}
\
\caption{Two new cohomologies, the associated primitive cohomologies, and their Laplacians (with $\la > 0$).  An additional primitive cohomology $PH_d^k(M)$ is introduced in Section \ref{homology}
\label{tabcoh}}
 \
 
\end{table}

The structure of this paper is as follows.  In Section 2, we begin by highlighting some special structures of differential forms on symplectic manifolds. This section will provide the foundation on which we build our analysis of symplectic cohomology.  In Section 3, we describe the de Rham and $H_{\dl}(M)$ cohomologies that have been studied on symplectic manifolds, and introduce the new cohomologies $H_{d+\dl}(M)$ and $H_{d\dl}(M)
$, and their associated primitive cohomologies.  We demonstrate the properties of these new cohomologies and also compare the various cohomologies in the context of the four-dimensional Kodaira-Thurston manifold.  In Section 4, we consider the identifying properties of dual currents of submanifolds in $M$.  We then show how the currents of coisotropic chains fit nicely within a primitive cohomology that can be considered dual to a homology of coisotropic chains.   We conclude in Section 5 by briefly comparing the cohomologies that we have constructed on symplectic manifolds with those previously studied in complex geometry.   

  Further discussion of the structures of primitive cohomologies and their applications will be given separately in a follow-up paper \cite{TY}.

\medskip

\noindent{\it Acknowledgements.~} {We would like to thank K.-W. Chan, J.-X. Fu, N.-C. Leung, T.-J. Li,  B. Lian, A. Subotic, C. Taubes, A. Todorov, A. Tomasiello, and V. Tosatti for helpful discussions.  We are also grateful to V. Guillemin for generously sharing his insights on this topic with us.  This work is supported in part by NSF grants 0714648 and 0804454.}

\


\section{Preliminaries}

In this section, we review and point out certain special structures of differential forms on symplectic manifolds.  These will provide the background for understanding the symplectic cohomologies discussed in the following sections.  For those materials covered here that are standard and well known, we shall be brief and refer the to the references \cite{Weil, Wells, LM, Brylinski, Yan, Guillemin, Huy, Caval}) for details.

Let $(M,\om)$ be a compact symplectic manifold of dimension $d=2n$.  Let $\CA^k(M,\mathbb{R})$ denote the space of smooth $k$-forms on $M$.  Using the symplectic form $\om = \sum \frac{1}{2}\, \om_{ij}\, dx^i \w dx^j$ (with summation over the indices $i,j$ implied), the Lefschetz operator $L: \CA^k(M) \to \CA^{k+2}(M)$ and the dual Lefschetz operator $\La: \CA^k(M) \to \CA^{k-2}(M)$ are defined acting on a $k$-form $A_k\in \CA^k(M)$ by
\begin{align*}
L &: ~~ L (A_k) =\om \w A_k ~,\\
\La &:~~ \La (A_k) =\frac{1}{2}(\om^{-1})^{ij}\,i_{\pa_{x^i}} i_{\pa_{x^j}} A_k~, 
\end{align*}
where $\w$ and $i$ denote the wedge and interior product, respectively, and $(\om^{-1})^{ij}$ is the inverse matrix of $\om_{ij}$.  In local Darboux coordinates $(p_1,\ldots, p_n, q_1,\ldots, q_n)$ where $\om= \sum dp_j \w dq_j$, we have
\begin{align*}
\La \,A_k = \sum i(\frac{\pa}{\pa q_j}) i(\frac{\pa}{\pa p_j}) ~ A_k~.
\end{align*}
$L$ and $\La$ together with the degree counting operator
\begin{equation*}
H =  \sum_k (n-k)\, \Pi^k~,
\end{equation*}
where $\Pi^k:\CA^*(M) \to \CA^k(M)$ projects onto forms of degree $k$, give a representation of the $sl(2)$ algebra acting on $\CA^*(M)$,
\be\label{Lalg}
[\La, L] = H\,, \qquad [H, \La]  = 2 \La\,,\qquad [H,L] = -2 L ~ ,
\ee
with the standard commutator definition $[a,b]:= a\, b -b \,a\,.$ 

Importantly, the presence of this $sl(2)$ representation allows for a ``Lefschetz" decomposition of forms in terms of irreducible finite-dimensional $sl(2)$ modules.  The highest weight states of these irreducible $sl(2)$ modules are the space of primitive forms, which we denote by $\CB^*(M)$.  

\begin{dfn}\label{pdef}
A differential $k$-form $B_k$ with $k\leq n$ is called {\it primitive}, i.e. $B_k \in \CB^k(M)$, if it satisfies the two equivalent conditions\,:~(i) $\La\, B_k = 0\,$; ~(ii) $L^{n-k+1}B_k=0$ .
\end{dfn}

Given any $k$-form, there is a unique Lefschetz decomposition into primitive forms \cite{Weil}.   Explicitly, we shall write 
\be\label{Lefdec}
A_k = \sum_{r\geq{\rm max}(k-n,0)}\frac{1}{r!}\, L^r B_{k-2r}~,
\ee
where each $B_{k-2r}$ can be written in terms of $A_k$ as 
\begin{align}\label{Bdef}
B_{k-2r} &= \Phi_{(k,k-2r)}(L,\La)\, A_k \notag\\
&\equiv \left(\sum_{s=0} \,a_{r,s}\,\frac{1}{s!}\,L^s\La^{r+s}\right) \, A_k~,
\end{align}
where the operator $\Phi_{(k,k-2r)}(L, \La)$ is a linear combination of $L$ and $\La$ with the rational coefficients $a_{r,s}$'s dependent only on $(d, k, r)$.  We emphasize that the Lefschetz decomposed forms $\{B_k, B_{k-2},\ldots\}$ are uniquely determined given a differential form $A_k$.  We give a simple example.

\begin{ex}For a four-form $A_4$ in dimension $d=2n =6\,$, the Lefschetz decomposed form is written as 
\bes\label{decompex}
A_4 = L \,B_2 + \frac{1}{2}\, L^2  \,B_0~.
\ees
Applying $\La$ and $\La^2$ to $A_4$ as written above and using the $sl(2)$ algebra in \eqref{Lalg}, the primitive forms $\{B_2, B_0\}$ are expressed in terms of $A_4$ as
\begin{align*}
B_2 &= \Phi_{4,2}(A_4)= (\La - \frac{1}{3}\, L \La^2 ) A_4 ~,\\
B_0 &=  \Phi_{4,0}(A_4)= \frac{1}{6}\,\La^2 A_4~.
\end{align*}
\end{ex}

\subsection{Three simple differential operators}

Three differential operators have a prominent role in this paper.  The first is the standard exterior derivative $d: \CA^k(M) \to \CA^{k+1}(M)$.  It interacts with the $sl(2)$ representation via the following commutation relations
\be\label{drels}
 [d,L] =0\, , \qquad [d,\La]= \dl\, ,\qquad~\, [d,H]=d ~.
\ee
The first and third relations follow trivially from $\om$ being symplectic and the definition of $H$, respectively.  We take the second relation to define the second differential operator\footnote{The $\dl$ operator we define here is identical to the $\delta$ operator in the literature, though some authors' definition of the $\La$ operator differs from ours by a sign (see for example, \cite{Brylinski, Caval}).  Our convention is that $\La\, L\,(f) = n\, f\,$, for $f$ a function.}  
\be\label{dLambda}
\dl := d \,\La - \La \, d~.
\ee
Notice in particular that $\dl: \CA^k(M) \to \CA^{k-1}(M)\,$, decreasing the degree of forms by one.   

Though not our emphasis, it is useful to keep in mind the original adjoint construction of $\dl$ \cite{EH, Libermann,Brylinski}.   Recall the symplectic star operator, $\ss: \CA^k(M) \to \CA^{2n-k}(M)$ defined by
\begin{align}\label{stardef}
A \w \ss A' &=  (\om^{-1})^k(A, A') \,d{\rm vol} \notag\\
& = \frac{1}{k!}(\om^{-1})^{i_1 j_1}(\om^{-1})^{i_2 j_2}\!\ldots (\om^{-1})^{i_k j_k}\, A_{i_1 i_2 \ldots i_k}\,A'_{j_1 j_2 \ldots j_k}\,\,  \frac{\om^n}{n!}~,
\end{align}
for any two $k$-forms $A, A' \in \CA^k(M)\,$.   This definition is in direct analogy with the Riemannian Hodge star operator where here $\om^{-1}$ has replaced $g^{-1}$.  Notice, however, that $\ss$ as defined in \eqref{stardef} does not give a positive-definite local inner product, as $A\w \ss A'$ is $k$-symmetric.  Thus, for instance, $A_k \w \ss A_k =0 $ for $k$ odd.  The symplectic star's action on a differential form can be explicitly written in terms of its action on each Lefschetz decomposed component $\frac{1}{r!} L^r B_s$ (as in \eqref{Lefdec} with $s=k-2r$).  It can be straightforwardly checked that \cite{Weil, Guillemin} 
\be\label{ssact}
\ss\; \frac{1}{s!}\, L^r B_s = (-1)^{\frac{s(s+1)}{2}} \frac{1}{(n-s-r)!}\,L^{n-s-r}\, B_s ~,
\ee
for $B_s \in \CB^s(M)$ and $r\leq n-s$.  This implies in particular
\bes
\ss \ss =1~.
\ees
 
The symplectic star operator permits us to consider $\La$ and $\dl$ as the symplectic adjoints of $L$ and $d$, respectively.  Specifically, we have the relations \cite{Yan}
\bes
\La =  \ss\, L\, \ss~
\ees
and \cite{Brylinski}
\be\label{dlsymp} 
\dl  = (-1)^{k+1} \ss d \ss ~,
\ee 
acting on $A_k\in \CA^k\,.$   Thus we easily find that $\dl$ squares to zero, that is, 
\bes\label{ddlanti}
\dl \dl = - \ss d^2 \ss = 0~.
\ees 
And by taking the symplectic adjoint of \eqref{drels}, we obtain the commutation relation of $\dl$ with the $sl(2)$ representation
\be\label{dlrels} 
[\dl, L] =d\, , \qquad [\dl, \La]=0\, ,  \qquad [\dl, H] = -\dl~.
\ee

The third differential operator of interest is the composition of the first two differential operators, $d\dl: \CA^k\to \CA^k\,.$  Explicitly,
\bes
d\dl = - d\, \La \, d = - \dl d~,
\ees
which implies in particular that $d$ and $\dl$ anticommute.  Besides not changing the degree of forms, $d\dl$ has a noteworthy property with respect to the $sl(2)$ operators.  Using  equations \eqref{drels} and \eqref{dlrels}, and also the commutation property $[ab, c]= a [b,c] + [a,c]b$, it is easily seen that $d\dl$ commutes with all three $sl(2)$ generators
\be\label{ddl}
[d\dl, L]= [d\dl, \La] = [d\dl, H] = 0~.
\ee
This implies in particular when it acts on primitive forms, $d\dl: \CB^k(M) \to \CB^k(M)\,$; that is, $d\dl$ preserves primitivity of forms.

To summarize, we write all the commutation relations together.
\begin{lem}\label{dddrels}The differential operators $(d, \dl, d\dl)$ satisfy the following commutation relations with respect to the $sl(2)$ representation $(L,\La, H)\,$:
\begin{alignat*}{2}
\tag{\ref{drels}}[d,L]=0\,,  ~~\qquad &[d,\La]=\dl\, ,\qquad &~~&[d,H]=d\, ,  \\
\tag{\ref{dlrels}} [\dl, L]=d\,, ~~\qquad &[\dl,\La]=0\, , \qquad &~~&[\dl,H]= -\dl\, , 
 \\
\tag{\ref{ddl}}[d\dl, L]=0\,, ~~\qquad &[d\dl,\La]=0\, , \qquad &~~&[d\dl,H] = 0 \, .
\end{alignat*}
\end{lem}

\subsection{$d$, $\dl$, and $d\dl$ acting on forms}

Consider now the action of the differential operators on a $k$-form written in Lefschetz decomposed form of \eqref{Lefdec}.  Straightforwardly, we have 
\bea
dA_k & = & \frac{1}{r!} \sum L^r \,d B_{k-2r} \label{dAk}~,\\
\dl A_k & = & \frac{1}{r!} \sum L^r \left( d B_{k - 2r -2} + \dl B_{k-2r}\right)\label{dlAk}~,\\
d\dl A_k & =& \frac{1}{r!} \sum L^r \,d\dl B_{k-2r}~, \label{ddlAk}
\eea
where $B_{k}\in \CB^{k}(M)$.  The first and third equations is simply due to the fact that $d$ and $d\dl$ commute with $L$.  The second follows from commuting $\dl$ through $L^r$ and repeatedly applying the relation that $[\dl, L] = d\,$.  

Lefschetz decomposing $dB_k$, we can formally write
\be\label{dBLef}
dB_k =  B^0_{k+1} + L \, B^1_{k-1} + \ldots + \frac{1}{r!} L^r B^r_{k+1-2r}\,+\ldots.
\ee
But in fact the differential operators acting on primitive forms have special properties.

\begin{lem} Let $B_k\in \CB^k(M)$ with $k\leq n$.  The differential operators $(d, \dl, d\dl)$ acting on $B_k$ take the following forms:
\begin{itemize}
\setlength{\parsep}{0pt}
\setlength{\itemsep}{0pt}
\item[{\rm(i)}] If $k<n$, then $d B_k = B^0_{k+1} + L \,B^1_{k-1}\,$; 
\item[{\rm(i')}] If $k=n$, then $d B_k= L\, B^1_{k-1}\,$;
\item[{\rm(ii)}] $\dl B_k = -H B^1_{k-1}=-(n-k+1) \,B^1_{k-1}\,$;
\item[{\rm(iii)}] $d\dl B_k = B^{01}_k=-(H+1)dB^1_{k-1}\,$;
\end{itemize}
for some primitive forms $B^0, B^1, B^{01}\in \CB^*(M)$.
\label{diffB}
\end{lem}
\begin{proof} (i) is the simple assertion that the Lefschetz decomposition of $dB_k$ contains at most two terms, or equivalently, $\La^2 dB =0\,$.  This follows from considering $\,0=\dl \La\, B = \La \dl B = - \La \La d B\,$, having used the relation $[\dl, \La]=0$.   
(i') removes the $B^0_{k+1}$ term on the right-hand side of (i) since primitive forms are at most of degree $n$.  
For (ii), it follows that
\begin{align*}\label{dlB}
\dl B_k &= - \La d \,B_k= - \La (L\, B^1_{k-1}) = -H B^1_{k-1} \notag\\
 & = -(n-k+1) B^1_{k-1}~,
\end{align*}
having used (i) and  the primitivity property $\La B_k = \La B^1_{k-1} =0$.  And as for (iii), the $d\dl$ operator preserves degree and commutes with $\La$.  Therefore, we must have $d\dl B_k = B^{01}_k$ and specifically $B^{01}_k=-(n-k+1)dB^1_{k-1}=-(H+1)dB^1_{k-1}$ applying $d$ to (ii).
\end{proof}

Let us recall the expression for the Lefschetz primitive forms:
\bes
B_{k-2r} = \Phi_{(k,k-2r)}(L,\La)\, A_k~. \tag{\ref{Bdef}}
\ees
This useful relation makes clear that $B_{k-2r}$ is explicitly just some combination of $L$ and $\La$ operators acting on $A_k$.  Since $d\dl$ commutes with $L$ and $\La$, this together with \eqref{ddlAk} implies the equivalence of the $d\dl$-closed and exact conditions on $A_k$ and its primitive decomposed forms $B_{k-2r}$.  Specifically, we have the following:
\begin{prop}Let $A_k, A'_{k}\in \CA^k(M)$.   Let $B_{k-2r}, B'_{k-2r}\in \CB^*(M)$ be respectively their Lefschetz decomposed primitive forms.  Then:
\begin{itemize}
\setlength{\parsep}{0pt}
\setlength{\itemsep}{0pt}
\item[{\rm(i)}] $d\dl$-closed: $d\dl A_k=0$ iff $\,d\dl B_{k-2r}=0$ for all $r$;
\item[\rm{(ii)}] $d\dl$-exact: $A_k= d\dl A'_{k}$ iff $\,B_{k-2r} = d\dl B'_{k-2r}$ for all $r$.
\end{itemize}
\label{ddlclex}
\end{prop}
\begin{proof} It follows straightforwardly from \eq{Bdef} and \eq{ddlAk} that $d\dl$ commutes with $L$ and $\La$.
\end{proof}
Note that similar type of statements cannot hold for $d$ or $\dl$, individually.  As seen in the commutation relations of Lemma 2.3, $d$ generates $\dl$ when commuted through $\La$, and $\dl$ generates $d$ when commuted through $L$.   But imposing $d$ and $\dl$ together, we have the closedness relation
\begin{prop}Let $A_k\in \CA^k(M)$ and $B_{k-2r}\in \CB^*(M)$ be its Lefschetz decomposed primitive forms.  Then $d A_k= \dl A_k =0$ if and only if $dB_{k-2r}=0\,$, for all $r$.
\label{dpdlcl}
\end{prop}
\begin{proof} Starting with \eq{Bdef}, we apply the exterior derivative $d$ to it.  Commuting $d$ through $\Phi_{(k,k-2r)}(L,\La)$, the $d$- and $\dl$-closedness of $A_k$ immediately implies $dB_{k-2r}=0$.   Assume now $dB_{k-2r}=0\,$, for all $B_{k-2r}$.  Note that this trivially also implies $\dl B_{k-2r}= -\La dB_{k-2r}=0$.  With the expressions \eq{dAk} and \eq{dlAk}, we therefore find $d A_k= \dl A_k =0$.
\end{proof}

In the proof, we have made use of the property that $dB_{k}=0$ implies $\dl B_k=0$.  Note, however, that the converse does not hold: $\dl B_{k}=0$ does not imply $dB_k=0$.

\subsection{Symplectomorphism and Lie derivative}

The three differential operators $d$, $\dl$, and $d\dl$ are good symplectic operators in the sense that they commute with all symplectomorphisms of a symplectic manifold.  Under a symplectomorphism, $\vp: (M,\om) \to (M,\om)$, the action on the constituents $d$ and $\La$ are 
\begin{align*}
\vp^*(dA)&= d (\vp^* A)~, \\
\vp^*(\La A) & = \La (\vp^*A)~,
\end{align*}
implying all three operators commute with $\varphi$.

Let us calculate how a differential form varies under a vector field $V$ that generates a symplectomorphism of $(M,\om)$.  The Lie derivative of $A_k$ follows the standard Cartan formula
\be\label{Cartan}
\CL_V\,  A_k = i_V (d\,A_k) + d(i_V A_k)~.
\ee
Since $V$ preserves $\om$, $\CL_V \om =0\,$ and there is a closed one-form associated to $V$,
\be\label{vtangent}
{ v= i_V \om ~, \qquad {\rm where}~~ dv=0~.}
\ee
Of interest, the Lie derivative of $V$ preserves the Lefschetz decomposition of forms and allows us to express the Lie derivative in terms of the $\dl$ operator.
\begin{lem}  \label{Lie}  Let $V$ be a vector field $V$ that generates a symplectomorphism of $(M,\om)$.  The action of the Lie derivative $\CL_V$ takes the form
\begin{itemize}
\setlength{\parsep}{0pt}
\setlength{\itemsep}{0pt}
\item[{\rm(i)}] $\CL_V A_k = \sum \frac{1}{r!} \,L^r  \, (\CL_V B_{k-2r})\,;$
\item[\rm{(ii)}] $\CL_V B_k = - \dl ( v\w B_k )  - v \w \dl B_k\,$.
\end{itemize}
\end{lem}
\begin{proof} (i) follows from $\CL_V \om =0\,$. It can also be shown by using Cartan's formula \eqref{Cartan} and the property $i_V(A_k \w A'_{k'}) = i_V(A_k)\w A'_{k'} + (-1)^k A_k \w i_V(A'_{k'})\,$.   For (ii), note that $V$ and $v$ in components are explicitly related by $V^i=(\om^{-1})^{ji}v_j$. Therefore, acting on a primitive form $B$, we have $i_V B = - \Lambda (v \w B)$.  Thus, we have
\begin{align*}
\CL_V B &= i_V (dB) + d(i_V B) = - \La ( v \w dB) + v \w (\La\, dB) - d\left[\La(v\w B)\right] \\
& = (\La d - d \La)(v \w B) + v \w (\La d)\, B \\
&= - \dl ( v\w B)  - v \w \dl B~.
\end{align*}
\end{proof}
 
In addition, if $V$ is a hamiltonian vector field - that is, $v=dh$ for some hamiltonian function $h$ - then the Lie derivative formula simplifies further in the following scenario. 
\begin{prop} \label{Liehdpdl} Let $A_k\in \CA^k(M)$ and $V$ a hamiltonian vector field with its associated one-form $v= dh$. If $dA_k=\dl A_k=0$, then 
\be\label{ddlh}
\CL_V A_k = d\dl (h\, A_k)~.
\ee
\end{prop}
\begin{proof} By Proposition \ref{dpdlcl}, $A_k$ being $d$- and $\dl$-closed implies that the Lefschetz decomposed primitive forms, $B_{k-2r}\,$, are $d$-closed for all $r$.  Now, $dB_{k-2r}=0$ implies $\dl B_{k-2r}=0\,$ and so  Lemma \ref{Lie}(ii) with $v=dh$ becomes 
\bes
\CL_V B_{k-2r} = -\dl(dh \w B_{k-2r}) = d\dl (h \,B_{k-2r})~.
\ees
The expression \eqref{ddlh} for $\CL_V A_k$ is then obtained using the commutativity of both the Lie derivative and $d\dl$ with respect to Lefschetz decomposition, as in Proposition \ref{Lie}(i) and \eqref{ddlAk}.
\end{proof}

That $d\dl$ naturally arise in the hamiltonian deformation of $d$- and $\dl$-closed differential forms is noteworthy, and we will make use of this property in the following sections in our analysis of cohomology.

\subsection{Compatible almost complex structure and Hodge adjoints}

Since the $\dl$ operator may not be as familiar, it is useful to have an alternative description of it.  We give a different expression for $\dl$ below making use of the compatible pair of almost complex structure and its associated metric which exists on all symplectic manifolds.  

An almost complex structure $J$ is said to be compatible with the symplectic form if it satisfies the conditions
\begin{align*}
\om(X,JX)&> 0 ~, \quad\quad  \forall \, X\neq 0 ~, \notag\\
\om(JX,JY)&=\om(X,Y)~. \qquad 
\notag
\end{align*}
These two conditions give a well-defined Riemannian metric $g(X,Y)= \om(X, JY),$ which is also hermitian with respect to $J$.  $(\om, J, g)$ together forms what is called a compatible triple.  We can use the metric $g$ to define the standard Hodge $*$ operator.  The dual Lefschetz operator $\La$ is then just the adjoint of $L$ (see for example \cite[p.\,33]{Huy})
\be\label{Linv}
\La =L^*=  (-1)^k * L\, *~.
\ee
We can write $\dl$ in terms of a standard differential operator in complex geometry.  
We will make use of the Weil relation for primitive $k$-forms $B_k$ \cite{Weil},
\be\label{Weil}
*\,\frac{1}{r!}L^r B_k = (-1)^{\frac{k(k+1)}{2}} \frac{1}{(n-k-r)!} \, L^{n-k-r} \CJ(B_k)~,
\ee
where 
\bes
\CJ=\sum_{p,q} (\sqrt{-1}\,)^{p-q} \ \Pi^{p,q}~
\ees
projects a $k$-form onto its $(p,q)$ parts times the multiplicative factor $(\sqrt{-1}\,)^{p-q}$.  Comparing \eqref{Weil} to the action of the symplectic star operator \eqref{ssact}, we have the relations
\be \label{srels}
* =  \CJ\, \ss ~.
\ee
This leads to following relation.
\begin{lem}\label{dldcs} Given any compatible triple $(\om, J, g)$ on a symplectic manifold, the differential operator $\dl = [d,\La]$ and the $d^c$ operator
\bes
d^c := \CJ^{-1} d\, \CJ~,
\ees
are related via the Hodge star operator defined with respect to the compatible metric $g$ by the relation  
\be\label{dldcdef}
 \dl = \dcs :=-* d^c * ~.
\ee
\end{lem}
\begin{proof} 
This can be shown directly starting from the definition of $\dl=d\La - \La d$ and making use of \eqref{Linv} and \eqref{Weil}, as in, for example, \cite[p.\,122]{Huy}.  Alternatively, we can write  
$\dl$ as the symplectic adjoint of $d$ and then apply \eqref{srels}.  Acting on a $k$-form, we have
\be
\dl = (-1)^{k+1} \ss d\, \ss = (-1)^{k+1} * \CJ^{-1} d \, *\, \CJ^{-1}  = (-1)^{k+1} * d^c * \CJ^{-2} =  \dcs ~,
\ee
having noted that $\CJ^{-2} = (-1)^k$, acting on a $k$-form.   
\end{proof}

Thus, making use of an almost complex structure, we have found that $\dl$ is simply $\dcs$.  We do emphasize that none of the formulas above requires the almost complex structure to be integrable.   In particular, $\dc\neq \sqrt{-1}(\bpa -\pa)$ in general.  From Lemma \ref{dldcs}, we again easily find $\dl\dl=0$ since $\dc\dc=0$.

Having expressed $\dl$ in terms of the Hodge star operator, we shall write down the standard Hodge adjoint of the differential operators.  With the inner product
\be\label{imet}
(A,A') =  \int_M A \w * A' = \int_M g(A, A')\, d{\rm vol} \, \qquad  A, A' \in \CA^k(M)~,
\ee
they are given as follows.
\begin{align} 
d^* &=  - * d * ~,\label{dstar}\\
\dls &= ([d,\La])^* =  [ L, d^*] = * \dl * ~,\label{dlstar} \\
(d\dl)^* &= - \ds L \,\ds = (-1)^{k+1} * d \dl * ~.\label{ddlstar}
\end{align}
The following commutation relations are easily obtained by taking the Hodge adjoints of the relations in Lemma \ref{dddrels}.
\begin{lem}\label{dssrels} For any compatible triple $(\om, J, g)$ on a symplectic manfiold, the commutation relations of the Hodge adjoints $(\ds,\dls,(d\dl)^*
)$ with the $sl(2)$ representation $(L,\La ,H)$ are 
\begin{align}
[\ds, L]& = - \dls\, , \qquad~\,\, [\ds, \La]=0\,,\qquad~~~~~[\ds, H] = -\ds\,, \\
[\dls, L]& =0 \, , \qquad\quad~~\, [\dls, \La] = - \ds \, ,\qquad [\dls, H] = \dls \,,\\
[(d\dl)^*, L]&=0\, , \quad\quad~\;   [(d\dl)^*,\La]=0\, , \qquad   [(d\dl)^*, H]=0\,.
\end{align} 
\end{lem}

\


\section{Symplectic cohomologies}

In this section, we discuss the cohomologies that can be constructed from the three differential operators $d$, $\dl$, and $d\dl$. We begin with the known cohomologies with $d$ (for de Rham $H_d$) and $\dl$ (for $H_{\dl}$).  In the relatively simple case of the $\dl$ cohomology, we show how a Hodge theory can be applied.  Then we consider building cohomologies by combining $d$ and $\dl$ together.

\subsection{$d$ and $\dl$ cohomologies}

With the exterior derivative $d$, there is of course the de Rham cohomology
\be\label{cohdR}
H^k_d(M) = \frac{\ker d \cap \CA^k(M)}{\im d \cap \CA^k(M)}~
\ee
that is present on all Riemannian manifolds.  Since $\dl\dl=0$, there is also a natural cohomology
\be\label{cohdl}
H^k_{\dl}(M)=\frac{\ker \dl \cap \CA^k(M)}{\im \dl \cap \CA^k(M)}~.
\ee
This cohomology has been discussed in \cite{Brylinski, Mathieu, Yan, Guillemin}.  

From the symplectic adjoint description of $\dl$, the two cohomologies can be easily seen to be related by the symplectic $\ss$ operator.   Expressing $\dl= (-1)^k \ss d\, \ss$, there is a bijective map given by $\ss$ between the space of $d$-closed $k$-form and the space of  $\dl$-closed $(2n-k)$-form.  For if $A_k$ is a $d$-closed $k$-form, then $\ss A_k$ is a $\dl$-closed $(2n-k)$-form.  Likewise, if $A_k = d A'_{k-1}$ is $d$-exact, then $(-1)^k \ss A_k = (-1)^k \ss d \ss (\ss A'_{k-1}) = \dl (\ss A'_{k-1})$ is $\dl$-exact. This implies the following proposition.
\begin{prop}[Brylinski \cite{Brylinski}] \label{ddliso} The $\ss$ operator provides an isomorphism between $H^k_d(M)$ and $H^{2n-k}_{\dl}(M)$.  Moreover, $\dim H^k_{d}(M) = \dim H^{2n-k}_{\dl}(M)$.
\end{prop} 

Proceeding further, we leave aside $\dl$'s symplectic adjoint origin and treat  it as an independent differential operator.  We utilize the compatible triple $(\om, J, g)$ on $M$ as discussed in Section 2.3 to write the Laplacian associated with the $\dl$ cohomology. 
\be\label{dlLap}
\Delta_{\dl} = \dls \dl + \dl \dls~,
\ee
where here $\dls$ is the Hodge adjoint in \eqref{dlstar}.  (Note that if we had used the symplectic adjoint, we would have obtained zero in the form of $d\dl + \dl d = 0\,$.)  The self-adjoint Laplacian naturally defines a harmonic form.  By the inner product, 
\bes
0=(A, \Delta_{\dl} A) = \| \dl A\|^2 + \| \dls A  \|^2 ~,
\ees
we are led to the following definition.

\begin{dfn}A differential form $A\in \CA^*(M)$ is called $\dl$-{\it harmonic} if  $\Delta_{\dl} A= 0$, or equivalently, $\dl A = \dls A =0$.   We denote the space of $\dl$-harmonic $k$-forms by $\CH_{\dl}^k(M)$.
\label{dl}
\end{dfn}

From Proposition \ref{ddliso}, we know that $H^k_{\dl}(M)$ is finite dimensional.   One may ask whether $\CH^k_{\dl}(M)$ is also finite dimensional. Intuitively, this must be so due to the isomorphism between $H^k_{\dl}(M)$ and $H^{2n-k}_{d}(M)$.  However, it will be more rewarding to address the question directly by calculating the symbol of the $\dl$ Laplacian.
\begin{prop}$\Delta_{\dl}$ is an ellipitic differential operator.
\end{prop}
\begin{proof}  To calculate the symbol of $\Delta_{\dl}$, we will work in a local unitary frame of $T^*M$ and choose a basis $\{\th^{1},\ldots, \th^{n}\}$ such that the metric is written as 
\bes
g= \th^i \otimes \bth^i + \bth^i \otimes \th^i~,
\ees
with $i=1,\ldots, n$.  
The basis one-forms satisfy the first structure equation, 
but as will be evident shortly, details of the connection one-forms and torsioin two-forms will not be relevant to the proof.  
With an almost complex structure $J$, any $k$-form can be decomposed into a sum of $(p,q)$-forms with $p+q=k$.  We can write a $(p,q)$-form in the local moving-frame coordinates
\bes
A_{p,q}= A_{i_1\ldots i_p j_1 \ldots j_q}~ \th^{i_1}\w \ldots \w \th^{i_p} \w  \bth^{j_1}\w \ldots \w \bth^{j_q}~.
\ees
The exterior derivative then acts as
\be\label{dApq}
d A_{p,q} \equiv (\pa A_{p,q})_{p+1,q} + (\bpa A_{p,q})_{p,q+1} + A_{i_1\ldots i_p j_1 \ldots j_q}~ d (\th^{i_1}\w \ldots \w \th^{i_p} \w  \bth^{j_1}\w \ldots \w \bth^{j_q})~,
\ee
where
\begin{align*}
 (\pa A_{p,q})_{p+1,q} &= \pa_{i_{p+1}}A_{i_1\ldots i_p j_1 \ldots j_q} ~ \th^{i_{p+1}}\w\th^{i_1}\w \ldots \w \th^{i_p} \w  \bth^{j_1}\w \ldots \w \bth^{j_q}~,\\
  (\bpa A_{p,q})_{p,q+1} &= \bpa_{j_{q+1}}A_{i_1\ldots i_p j_1 \ldots j_q} ~ \bth^{j_{q+1}}\w\th^{i_1}\w \ldots \w \th^{i_p} \w  \bth^{j_1}\w \ldots \w \bth^{j_q}~.
\end{align*}
In calculating the symbol, we are only interested in the highest-order differential acting on $A_{i_1\ldots i_p j_1 \ldots j_q} $.  Therefore, only the first two terms of (\ref{dApq}) are relevant for the calculation.  In dropping the last term, we are effectively working in $\mathbb{C}^n$ and can make use of all the K\"ahler identities involving derivative operators.  So, effectively, we have (using $\simeq$ to denote equivalence under symbol calculation) 
\bea\label{dreduce}
d&\simeq&\pa + \bpa ~,\notag\\
\dl= d^c{}^*&\simeq&  \sqrt{-1}(\bpa-\pa)^*=\sqrt{-1} (\pas - \bpas)~,
\eea
where we have used the standard convention, $\pas = - * \bpa *$ and $\bpas = - * \pa *$.  We thus have
\bea\label{dlsymbol}
\Delta_{\dl} &= & \dls \dl + \dl \dls = \dc \dcs + \dcs \dc \notag\\
& \simeq & (\pa - \bpa)(\pas - \bpas) + (\pas - \bpas)(\pa - \bpa) \notag\\
& \simeq &  \pa \pas + \pas \pa + \bpa \bpas + \bpas \bpa \notag\\
& \simeq & \Delta_{d}~,
\eea
where $\Delta_{d}=\ds d + d \ds$ is the de Rham Laplacian.  Clearly then, $\Delta_{\dl}$ is also elliptic.
\end{proof}

Applying elliptic theory to the $\dl$ Laplacian then implies the Hodge decomposition
\bes
\CA^k = \CH^k_{\dl} \oplus  \dl \CA^{k+1} \oplus \dls \CA^{k-1}~.
\ees
Moreover, $\ker \dl = \CH_{\dl} \oplus \im \dl$ and $\ker \dls = \CH_{\dl} \oplus \im \dls$.  Therefore, every $H_{\dl}$ cohomology class contains a unique $\dl$ harmonic representative and $\CH^k_{\dl}\cong H^k_{\dl}(M)\,.$  We note that although the explicit forms of the harmonic representative depend on $g$, the dimensions $\dim\CH^k_{\dl}(M)= \dim H^k_{\dl}(M)$ are independent of $g$ (or $J$).  In fact, by Proposition \ref{ddliso}, we have $\dim\CH^k_{\dl} = b_{2n-k}$ where $b_k$ are the $k${th}-Betti number.

As is clear from Proposition \ref{ddliso}, the $H_{\dl}(M)$ cohomology does not lead to new invariants as it is isomorphic to the de Rham cohomology.  We have however demonstrated that with the introduction of a compatible triple $(\om, J, g)$, it is sensible to discuss the Hodge theory of $H_{\dl}(M)$.  In the following, we shall apply methods used here to study new symplectic cohomologies.

\subsection{$d+\dl$ cohomology}

Having considered the cohomology of $d$ and $\dl$ operators separately, let us now consider forms that are closed under both $d$ and $\dl$, that is, $dA_k =  \dl A_k=0$.  These forms were called symplectic harmonic by Brylinski \cite{Brylinski}.  Notice that any form that is $d\dl$-exact, $A_k=d\dl A'_k$ are trivially also $d$- and $\dl$-closed.  This gives a differential complex 
\begin{equation}\label{diffplex}
\xymatrix@1{
\CA^k \; \ar[r]^{d\dl} &\;  \CA^k\;  \ar[r]^-{d + \dl} & ~\CA^{k+1} \oplus \CA^{k-1}~.
}
\end{equation}
Considering the cohomology associated with this complex leads us to introduce 
\be\label{coha}
H^k_{d+\dl}(M) = \frac{\ker (d + \dl) \cap \CA^k(M)}{{\rm im} \ d\dl \cap \CA^k(M)}~.
\ee
Such a cohomology may depend on the symplectic form but is otherwise invariant under a symplectomorphism of $(M,\om)$.  It is also a natural cohomology to define with respect to hamiltonian actions.  By Proposition \ref{Liehdpdl}, the Lie derivative with respect to a hamiltonian vector field of a differential form that is both $d$- and $\dl$-closed is precisely $d\dl$-exact.  Hence, we see that the $H^k_{d+\dl}(M)$ cohomology class is invariant under hamiltonian isotopy.

Of immediate concern is whether this new symplectic cohomolgy is finite dimensional.  We proceed as for the $H_{\dl}(M)$ case by considering the corresponding Laplacian and its ellipiticity.  

From the differential complex, the Laplacian operator associated with the cohomology is  
\be\label{cohah}
\Delta_{d+\dl} = d \dl (d \dl)^* + \la (\ds d + \dls \dl)~.
\ee
where we have inserted an undetermined real constant $\la >0$ which gives the relative weight between the terms.  With the presence of $d\dl$ term in the cohomology, the Laplacian becomes a fourth-order differential operator.   By construction, the Laplacian is self-adjoint, so the requirement
\bes
0=(A, \Delta_{d+\dl}A) = \|(d\dl)^* A\|^2 + \la (\| d A\|^2 + \| \dl A\|^2) ~.
\ees
give us the following definition.

\begin{dfn}A differential form $A\in \CA^*(M)$ is called {\it $d+\dl$-harmonic} if  $\Delta_{d+\dl} A= 0$, or equivalently, 
\be\label{cohahc}
dA =\dl A = 0 ~~\quad {\rm and} \qquad  (d\dl)^* A =0~.
\ee
We denote the space of $d+\dl$-harmonic $k$-forms by $\CH_{d+\dl}^k(M)$.
\label{dpdl}
\end{dfn}

We now show that $H^*_{d+\dl}(M)$ is finite dimensional by analyzing the space of its harmonic forms.
\begin{thm}  Let $M$ be a compact symplectic manifold.  For any compatible triple $(\om, J, g)$, we define the standard inner product on $\CA^k(M)$ with respect to $g$.  Then:
\begin{itemize}
\setlength{\parsep}{0pt}
\setlength{\itemsep}{0pt}
\item[{\rm(i)}]
$\dim \CH_{d+\dl}^k(M) < \infty\,.$
\item[{\rm(ii)}] There is an orthogonal decomposition
\be\label{cohadecomp}
\CA^k = \CH^k_{d+ \dl}  \oplus d \dl \CA^{k} \oplus (\ds \CA^{k+1} + \dls \CA^{k-1})\,.
\ee
\item[{\rm(iii)}] There is a canonical isomorphism: $\CH^k_{d+\dl}(M) \cong H^k_{d+\dl}(M)$\,.
\end{itemize}
\label{dpdlellip}
\end{thm}

\begin{proof}
One can try to prove finiteness by calculating the symbol of $\Delta_{d+\dl}$.  This turns out to be inconclusive as the symbol of $\Delta_{d+\dl}$ is not positive.  However, we can introduce a related fourth-order differential operator that is elliptic.  Consider the self-adjoint differential operator
\be\label{saO}
D_{d+\dl} = (d \dl)(d\dl)^* + (d\dl)^*(d\dl) + d^*\dl \dls d + \dls d d^* \dl +\la( d^* d +  \dls \dl)~,
\ee
with $\la >0$.  Although $D_{d+\dl}$ contains three additional fourth-order differential terms compared with $\Delta_{d+\dl}$, the solution space of $D_{d+\dl} A =0$ is identical to that of $\Delta_{d+\dl}A=0$ in \eqref{cohahc}.  For consider the requirement,
\bes
0=(A, D_{d+\dl}A) = \|(d\dl)^* A\|^2 + \|d\dl A \|^2 + \|\dls d A \|^2 + \| d^* \dl A\|^2+ \la(\| d A\|^2 + \| \dl A\|^2) ~.
\ees
The three additional terms clearly do not give any additional conditions and are automatically zero by the requirement $dA = \dl A =0\,$.  Essentially, the presence of the two second-order differential terms ensures that the solutions space of $\Delta_{d+\dl}\, A =0 $ and $D_{d+\dl}\,A =0 $ match exactly. 

We now show that $D_{d+\dl}$ is elliptic.  The symbol calculation is very similar to that for the $H_{\dl}(M)$ cohomology in \eqref{dlsymbol}.   Keeping only fourth-order differential terms, and again using $\simeq$ to denote equivalence of the symbol of the operators, we find
\begin{eqnarray}\label{saOs}
D_{d+\dl}&\simeq&d \dl \dls \ds + \dls \ds d \dl + \ds \dl \dls d + \dls d \ds \dl\notag \\
&\simeq & d\ds \dl \dls + \ds d \dls \dl + \ds d \dl \dls + d \ds \dls \dl \notag\\
& \simeq & ( \ds d + d \ds ) (\dls \dl + \dl \dls) = \Delta_d \Delta_{\dl} \notag\\
&\simeq & \Delta^2_d~.
\end{eqnarray}
In the above, as we had explained for the calculations for $\Delta_{\dl}$ in \eqref{dlsymbol}, only the highest-order differential needs to be kept for computing the symbol, and so we can freely make use of K\"ahler identities.  And indeed, we used in line two and three the K\"ahler relations (with $\dl = \dcs$ from Lemma \ref{dldcs})
\bes
\dl \ds \simeq - \ds \dl\, and \qquad \dls d \simeq - d \dls~,
\ees
and $\Delta_{\dl}\simeq\Delta_d$ in line four of \eqref{saOs}.  In all, the symbol of $D_{d+\dl}$ is equivalent to that of the square of the de Rham Laplacian operator.  $D_{d+\dl}$ is thus elliptic and hence its solution space, which consists of $\CH^k_{d+\dl}(M)\,$, is finite dimensional.  

With $D_{d+\dl}$ elliptic, assertion (ii) then follows directly by applying elliptic theory.  For (iii), using the decomposition of (ii), we have $\ker (d + \dl)= \CH^{d+\dl} \oplus \,\im d\dl$.  This must be so since if $\ds A_{k+1} + \dls A_{k-1}$ is $d$- and $\dl$-closed, then
\begin{align*}
0&=(A_{k+1}, d(\ds A_{k+1} + \dls A_{k-1})  + (A_{k-1},\dl(\ds A_{k+1} + \dls A_{k-1}))\\
&= (\ds A_{k+1} + \dls A_{k-1},\ds A_{k+1} + \dls A_{k-1})\\
&=\|\ds A_{k+1} + \dls A_{k-1}\|^2~.
\end{align*}
Thus, every cohomology class of $H_{d+\dl}(M)$ contains a unique harmonic representative and $\CH^k_{d+\dl}(M)\cong H^k_{d+\dl}(M)$.  
\end{proof}
\begin{cor} For $(M,\om)$ a compact symplectic manifold, $\dim H^k_{d+\dl}(M) < \infty\,$.
\end{cor}

In short, we have been able to apply Hodge theory to $H_{d+\dl}(M)$ by equating the harmonic solution space with those of $D_{d+\dl}\,$, which we showed is an elliptic operator.  Having demonstrated that $H_{d+\dl}(M)$ is finite dimensional, we shall proceed to consider some of its properties.  

\subsubsection{Lefschetz decomposition and $\dpdl$ primitive cohomology}

Consider the Lefschetz decomposition reviewed in Section 2. It is generated by the $sl(2)$ representations $(L,\La, H)$.  Let us note that the $d+\dl$ Laplacian has the following special property: 
\begin{lem} \label{dpdllac}
$\Delta_{d+\dl}$ commutes with the $sl(2)$ triple $(L,\La, H)$.  
\end{lem}
\begin{proof} Since $\Delta_{d+\dl}: \CA^k(M) \to \CA^k(M)$ preserves the degree of forms, $[\Delta_{d+\dl}, H]=0$ is trivially true.  The commutation relations of $L$ and $\La$ follow from Lemma \ref{dddrels} and Lemma \ref{dssrels} 
\begin{align*}
[\Delta_{d+\dl},L]&=\la[d^* d +  \dls \dl, L] = \la\left([d^*,L] d + \dls [\dl, L]\right) =0 ~,\\
[\Delta_{d+\dl},\La]&=\la[d^* d +  \dls \dl, \La] = \la \left(d^*[d,\La] + [\dls , \La] \dl\right) = 0~,
\end{align*}
having noted that both $d\dl$ and $(d\dl)^*$ commute with $L$ and $\La$. 
\end{proof}

It is worthwhile to point out that in contrast, the de Rham Laplacian and the $\dl$ Laplacian $\Delta_{\dl}$ \eqref{dlLap} do not by themselves commute with either $L$ or $\La$.   In fact, the elliptic operator $D_{d+\dl}$ also does not commute with $L$ and $\La$.  That $\Delta_{d+\dl}$ commute with the $sl(2)$ representation is rather special.  It also immediately implies the following.
\begin{cor}On a symplectic manifold of dimension $2n$, and a compatible triple $(\om, J, g)$, the Lefschetz operator defines an isomorphism
\bes  L^{n-k}:\ \ \CH_{d+\dl}^k \cong \CH_{d+\dl}^{2n-k}\qquad {\rm for~} k \leq n~.
\ees 
\end{cor}
Proposition \ref{dpdllac} and the isomorphism of $\CH_{d+\dl}(M)$ with $H_{d+\dl}(M)$ points to a Lefschetz decomposition of the $d+\dl$ cohomology.  This can also be seen directly from the definition of the $d+\dl$ cohomology.  Recall from Proposition \ref{dpdlcl} that the condition $\ker d \cap \ker \dl$ is equivalent to all the Lefschetz decomposed forms be $d$-closed.   And Proposition \ref{ddlclex}(ii) implies that the primitive components of $d\dl$-exact forms are also $d\dl$-exact.  This leads us to following primitive cohomology 
\be\label{cohad}
PH^k_{\dpdl}(M) = \frac{\ker d \cap P^k(M)}{\im d\dl \cap P^k(M)}= \frac{\ker d \cap P^k(M)}{d\dl \CB^k(M)}~.
\ee
The second equivalence is non-trivial.  It is the statement that this cohomology can be considered as a cohomology {\it purely} on the space of primitive forms $\CB^*(M)$.  This equivalence holds by the following lemma.
\begin{lem} \label{ddlprim}On a symplectic manifold $(M,\om)$, if $B_k\in \CB^k(M)$ is $d\dl$-exact, then there exists a $B'_k\in \CB^k(M)$ such that $B_k =d\dl B'_k$. 
\end{lem}
\begin{proof} That $B_k$ is $d\dl$-exact means only that $B_k= d\dl A'_k$ for some $A'_k\in \CA^k(M)$.  Lefschetz decomposing $A'_k$ and imposing the primitivity condition give
\begin{align*}
0=\La B_k& = \La\, d\dl A'_k = d\dl \sum_{r\geq 0} \frac{1}{r!} \La\, L^r B'_{k-2r} \\
& = d\dl \sum_{r\geq 1}\frac{1}{r!}L^{r-1}(n-k+r+1) B'_{k-2r} \\
& = \sum_{s\geq 0} \frac{1}{s!} L^s \left[ (n-k+s+2)\, d\dl B'_{k-2-2s}\right] ~,
\end{align*}
where in the second line we have used the identity $[\La, L^r] = r \,L^{r-1}(H-r+1)$ and in the third line have set $s=r-1\,$.  The last line can be interpreted as the Lefschetz decomposition of a $(k-2)$-form $A''_{k-2}$ with primitive components $B''_{k-2-2s}=(n-k+s+2)\,d\dl B'_{k-2-2s}$.  But since $A''_{k-2}=0$ and $(n-k+s+2) > 0$ (since $k\leq n$), the uniqueness of Lefschetz decomposition for the form $A''_{k-2}=0$ implies $d\dl B'_{k-2-2s}=0$ for all $s\geq 0$.  Thus, we find $B_k= d\dl A'_k = d\dl B'_k$ where $B'_k$ is the primitive part of $A'_k = B'_k + L B'_{k-2r} + \ldots.$  
\end{proof}

For this primitive cohomology, we also have primitive harmonic forms that follow directly by imposing $\Delta_{d+\dl} B = 0$.  Because the forms are now primitive, the Laplacian simplifies to
\be\label{cohpa}
\Delta^p_{d+\dl} = d \dl (d \dl)^* + \la\, \ds d~,
\ee  
with $\la >0\,$.  Note that $\Delta^p_{d+\dl}B=0$ implies $\Delta_{d+\dl} B=0$ since for a primitive form, $dB=0$ implies $\dl B=0$ (see Lemma \ref{diffB}).  Thus, we define

\begin{dfn}A differential form $B\in \CB^*(M)$ is called {\it $(d+\dl)$-primitive harmonic} if  $\Delta^p_{d+\dl} B= 0$, or equivalently, 
\be\label{dprimh}
dB =0~ , \qquad  (d\dl)^* B =0~.
\ee
We denote the space of $(d+\dl)$-primitive harmonic $k$-forms by $P\CH_{\dpdl}^k(M)\,$.
\label{dprim}
\end{dfn}

\begin{thm}On a compact symplectic manifold $(M,\om)$ of dimension $2n$, $H^*_{d+\dl}(M)$ satisfies the following properties:
\setlength{\parsep}{0pt}
\begin{itemize}
\setlength{\itemsep}{0pt}
\item[{\rm(i)}] There is a Lefschetz decomposition
\begin{align}
H^k_{d+\dl}(M)& = \bigoplus_r\, L^r\, PH_{\dpdl}^{k-2r}(M)~,\notag\\
\CH^k_{d+\dl}(M)&=\bigoplus_r\, L^r \, P\CH_{\dpdl}^{k-2r}(M)~.\notag
\end{align} 
\item[{\rm(ii)}] Lefschetz property: the Lefschetz operator defines an isomorphism 
\bes
L^{n-k}:\ \ H_{d+\dl}^k(M) \cong H_{d+\dl}^{2n-k}(M) \qquad {\rm for~} k \leq n~.
\ees
\end{itemize}
\label{dpdlLef}
\end{thm}
\begin{proof} Assertion (i) follows from Lefschetz decomposing the forms and applying Proposition \ref{dpdlcl} and Proposition \ref{ddlclex} as discussed above.  (ii) follows from (i) by applying $L^{2n-k}$ to the Lefschetz decomposed form $H_{d+\dl}^k(M)\,$.
\end{proof}

\subsubsection{de Rham cohomology and the $d\dl$-lemma} 

We now explore the relationship between $d+\dl$ cohomology and de Rham cohomology $H_d(M)$.  There is a canonical homomorphism $H^k_{d+\dl}(M) \to H_d^k(M)$.  Trivially, the space of $d$- and $\dl$-closed forms is a subset of $d$-closed forms, and, likewise, the space of $d\dl$-exact forms is a subset of $d$-exact.   However, the mapping between the two cohomologies is neither injective nor subjective.  A trivial class in $H^k_{d+\dl}(M)$ certainly maps to a trivial class in $H_d^k(M)\,$,  but a non-trivial class in $H^k_{d+\dl}(M)$ can be trivial in $H_d^k(M)$.  For instance, a $d$-exact form  can be non-trivial in $H^k_{d+\dl}(M)$, since in general, it may not be also $d\dl$-exact.  A $d$-exact form is only always $d\dl$-exact if the below $d\dl$-lemma holds.

\begin{dfn} {\bf ($d\dl$-lemma)}  Let $A$ be a $d$- and $\dl$-closed differential form.  We say that the $d\dl$-lemma holds if the following properties are equivalent:
\begin{itemize}
\setlength{\itemsep}{0pt}
\item[{\rm(i)}] $A$ is $d$-exact;
\item[{\rm(ii)}] $A$ is $\dl$-exact;
\item[{\rm(iii)}] $A$ is $d\dl$-exact.
\end{itemize}
\label{ddllem}
\end{dfn}

\noindent Without the $d\dl$-lemma, the canonical mapping $H^k_{d+\dl}(M) \to H^k_d(M)$ is generally not injective.  As for surjectivity, a de Rham cohomology class need not have a representative that is also $\dl$-closed.  As mentioned in the Introduction, Matheiu \cite{Mathieu} (see also \cite{Yan}) identified the existence of a $d$- and $\dl$-closed form in every de Rham class with the strong Lefschetz property, which is not satisfied by every symplectic manifold.  Interestingly, as shown by Merkulov \cite{Merkulov} and Guillemin \cite{Guillemin} (see also \cite{Caval}), the existence of the $d\dl$-lemma on a compact symplectic manifold $M$ is equivalent to $M$ having the strong Lefschetz property.  Therefore, this implies:
\begin{prop} On a compact symplectic manifold $(M,\om)$, the $d\dl$-lemma holds, or equivalently, the strong Lefschetz property is satisfied, if and only if the canonical homomorphism $H^k_{d+\dl}(M) \to H_d^k(M)$ is an isomorphism for all $k$.
\label{ddllemprop}
\end{prop}
\begin{proof}Assuming first the $d\dl$-lemma.  Injectivity is then assured.  Surjectivity follows from the existence in each de Rham cohomology class of a $d$- and $\dl$-closed representative when the strong Lefschetz property holds (Mathieu's theorem \cite{Mathieu}).  Conversely, using again Mathieu's theorem, if the map is surjective, so that each de Rham class has a representative that is also $\dl$-closed, then the strong Lefschetz holds.   
\end{proof}

If the canonical homomorphism is an isomorphism, then the dimensions of the two cohomologies are equal.  We thus have the following corollary.

\begin{cor}On a compact symplectic manifold $(M,\om)$, if the $d\dl$-lemma holds, or equivalently if the strong Lefschetz property is satisfied, then $\dim H^k_{d+\dl}(M) = \dim H_d^k(M)\,$ for all $k$.
\end{cor}


\subsection{$d\dl$ cohomology}

The $d+\dl$ cohomology followed from the short differential complex
\begin{equation*}
\xymatrix@1{
\CA^k \; \ar[r]^{d\dl} &\;  \CA^k\;  \ar[r]^-{d + \dl} & ~\CA^{k+1} \oplus \CA^{k-1}~.
}
\end{equation*}
Interestingly, simply reversing the arrows gives another differential complex
\begin{equation*}\label{diffplex2}
\xymatrix@R=0pt{
\CA^{k+1} \ar[dr]^ \dl \\
\oplus & \CA^k \ar[r]^-{d\dl} & \CA^k ~.\\
\CA^{k-1} \ar[ur]_{d} & & 
}
\end{equation*}
This leads us to introduce the $d\dl$ cohomology
\be\label{cohb}
H^k_{d\dl}(M)= \frac{\ker d\dl \cap \CA^k}{(\im d + \im \dl)\cap \CA^k}~.
\ee
As may be expected, this $d\dl$ cohomology is closely related to the $d+\dl$ cohomology.  In fact, they are dual to each other, as we will explain below in Section \ref{ddldual}.  For now, let us proceed to describe some its properties.  

The associated Laplacian operator of the $d\dl$ cohomology is also fourth-order
\be\label{cohbh}
\Delta_{d\dl} = (d \dl)^* d \dl + \la (d\ds + \dl \dls)~
\ee
where $\la >0$.  A harmonic form of this Laplacian satisfies 
\bes
0=(A, \Delta_{d\dl}A) = \|d\dl A\|^2 + \la \left(\| \ds A\|^2 + \| \dls A\|^2\right) ~,
\ees
which leads us to the following definition.

\begin{dfn}A differential form $A\in \CA^*(M)$ is called {\it $d\dl$-harmonic} if  $\Delta_{d\dl} A= 0$, or equivalently, 
\be\label{cohbhc}
d\dl A =0~ , \qquad  \ds A = 0~ ,\qquad \dls A =0~.
\ee
We denote the space of $d\dl$-harmonic $k$-forms by $\CH_{d\dl}^k(M)\,$.
\label{dadl}
\end{dfn}

Similar to $\Delta_{d+\dl}$, $\Delta_{d\dl}$ is also not an elliptic operator.  But as before, we can consider an elliptic operator whose space of solution is identical to that of $\Delta_{d\dl}$.  Let
\be\label{sbO}
D_{d\dl} = (d\dl)^*(d\dl) + (d \dl)(d\dl)^* + d\dls \dl \ds + \dl \ds d \dls +\la( d \ds +  \dl \dls)~.
\ee
Let us show that $D_{d\dl}$ is elliptic using the same method of the previous subsections.  Again labeling symbol equivalence by $\simeq$, we have
\begin{eqnarray*}\label{sbOs}
D_{d\dl}&\simeq&\dls \ds d \dl + d \dl \dls \ds +  d \dls \dl \ds + \dl \ds d \dls \\
&\simeq & d\ds \dl \dls + \ds d \dls \dl + \ds d \dl \dls + d \ds \dls \dl \\
& \simeq & ( \ds d + d \ds ) (\dls \dl + \dl \dls) = \Delta_d \Delta_{\dl} \\
&\simeq & \Delta^2_d~.
\end{eqnarray*}
Then by applying ellliptic theory arguments, we have proved
\begin{thm}  Let $M$ be a compact symplectic manifold.  For any compatible triple $(\om, J, g)$, we define the standard inner product on $\CA^k(M)$ with respect to $g$.  Then:
\begin{itemize}
\setlength{\parsep}{0pt}
\setlength{\itemsep}{0pt}
\item[{\rm(i)}]
$\dim \CH_{d\dl}^k(M) < \infty$~.
\item[{\rm(ii)}] There is an orthogonal decomposition
\be\label{cohbdecomp}
\CA^k = \CH^k_{d\dl}   \oplus (d\CA^{k-1} + \dl \CA^{k+1}) \oplus (d\dl)^* \CA^{k}~.
\ee
\item[{\rm(iii)}] There is a canonical isomorphism: $\CH^k_{d\dl}(M) \cong H^k_{d\dl}(M)$\,.
\end{itemize}
\label{dadlellip}
\end{thm}
Also,
\begin{cor} For $(M,\om)$ a compact symplectic manifold, $\dim H^k_{d+\dl}(M) < \infty\,$.
\end{cor}

\subsubsection{Lefschetz decomposition and the $d\dl$ primitive cohomology}

Like the $d+\dl$ cohomology, the $d\dl$ cohomology exhibits Lefschetz decomposition.  The arguments are very similar to those in section 3.2.1, so we will mainly state results here and focus on the differences.   

The starting point is again to note: 
\begin{lem} \label{ddllac}
$\Delta_{d\dl}$ commutes with the $sl(2)$ triple $(L,\La, H)$.  
\end{lem}
This implies the following corollary:
\begin{cor}On a symplectic manifold of dimension $2n$, and a compatible triple $(\om, J, g)$, the Lefschetz operator defines an isomorphism
\bes  L^{n-k}:\ \ \CH_{d\dl}^k(M) \cong \CH_{d\dl}^{2n-k}(M)\qquad {\rm for~} k \leq n~.
\ees 
\end{cor}
Hence, this leads to a Lefschetz decomposition of the $d\dl$-harmonic forms and also of $H^*_{d\dl}(M)$ by the canonical isomorphism of Theorem \ref{dadlellip}(iii).  The $d\dl$ primitive cohomology takes the form
\be\label{cohbd}
PH^k_{d\dl}(M) = \frac{\ker d\dl \cap \CB^k(M)}{(\im d + \im \dl) \cap \CB^k(M)}=\frac{\ker d\dl \cap \CB^k(M)}{ (d + L\, H^{-1} \dl) \CB^{k-1} + \dl \CB^{k+1}}~,
\ee
where $H^{-1} = \sum_k \dfrac{1}{n-k}\, \Pi^k$ is the inverse of the degree- counting operator $H$.   Here, we have again written the cohomology as defined solely on the space of primitive forms.  This is possible by the following property of $d+\dl$-exact primitive form.

\begin{lem} \label{dadlprim}On a symplectic manifold $(M,\om)$, if $B_k\in \CB^k(M)$ is $d+\dl$-exact - that is, $B_k = d A'_{k-1} + \dl A''_{k+1}$ - then there exist two primitive forms $\hB'_{k-1}$ and $\hB''_{k+1}$ such that $B_k =(d + L H^{-1} \dl) \hB'_{k-1} + \dl \hB''_{k+1}$. 
\end{lem}
\begin{proof}Since $B_k$ is primitive, $k\leq n$.  Lefschetz decomposing $A'_{k-1}$ and $A''_{k+1}$, we have
\bes
A'_{k-1} = B'_{k-1} + \ldots  \ , \qquad A''_{k+1}= B''_{k+1} + L B''_{k-1} + \ldots~,
\ees
with $B''_{k+1} =0$ if $k=n\,$.  In the above, we have only written out the first few terms of the decomposition; the other terms will not play a role.  Now 
\begin{align*}
d \, A'_{k-1} &= (B'^0_{k} + L B'^1_{k-2}) + \ldots ~, \\
\dl\, A''_{k+1} & = -H B''^1_{k} + (B''^0_k + L B''^1_{k-2} - L H B''^1_{k-2})+ \ldots~, 
\end{align*}
where superscripted primitive forms $B^0_{s+1}$ and $B^1_{s-1}$ are the primitive components of $dB_s = B^0_{s+1} + L\, B^1_{s-1}$ as in Lemma \ref{diffB}(i).  To write down the second equation, we have used the relations $[\dl, L]= d$ in Lemma \ref{dddrels} and $\dl B_s = - H B^1_{s-1}$ of Lemma \ref{diffB}(ii).  
Since $dA'_{k-1} + \dl A''_{k+1}$ is primitive, we must have
\bes
B_k = dA'_{k-1} + \dl A''_{k+1} = B'^0_k + B''^0_k - H B''^1_{k} ~,
\ees
with any additional terms on the right with powers of $L$ vanishing.  Let us now define
\bes
\hB'_{k-1}= B'_{k-1} + B''_{k-1}~, \qquad \hB''_{k+1} = B''_{k+1} ~.
\ees
We only need to check
\begin{align*}
(d + L H^{-1} \dl)\hB'_{k-1} + \dl \hB''_{k+1}& = \hB'^0_k - H \hB''^1_{k+1} \\
&= (B'^0_{k} + B''^0_{k}) - H B''^1_k =B_k~,
\end{align*}
noting that in general the differential operator $(d + L H^{-1} \dl) : B_s \to B^0_{s+1}\,$. \footnote{Acting on a primitive form, $(d + L H^{-1} \dl)B_s = (1- L H^{-1} \La) d B_s=B^0_{s+1}\,$.  Hence, this operator consists of the exterior derivative followed by a projection onto the first primitive component of the Lefschetz decomposition.  Properties of this operator will be discussed more fully in \cite{TY}.}   
\end{proof}

Let us now describe the harmonic forms of the $d\dl$ primitive cohomology. We shall impose $\Delta_{d\dl} B = 0$.  But with the forms now primitive, the Laplacian simplifies to
\be\label{cohpb}
\Delta^p_{d\dl} = (d \dl)^* d \dl + \la\, \dl \dls~,
\ee  
with $\la >0\,$.  In the above, we have noted that for primitive forms, $\dls B =0$ implies $\ds B=0$.  This follows from the commutation relation $[\dls, \La] = - \ds$ in Lemma \ref{dssrels}.   Let us define the following:

\begin{dfn}A differential form $B\in \CB^*(M)$ is called {\it $d\dl$-primitive harmonic} if  $\Delta^p_{d\dl} B= 0$, or equivalently, 
\be\label{pcohc}
d\dl B =0\, , \qquad  \dls B =0~.
\ee
We denote the space of $d\dl$-primitive harmonic $k$-forms by $P\CH_{d\dl}^k(M)$.
\label{ddlprimh}
\end{dfn}

We collect the Lefschetz decomposition properties.

\begin{thm}On a compact symplectic manifold $(M,\om)$ of dimension $2n$, $H^*_{d\dl}(M)$ satisfies the following properties:
\setlength{\parsep}{0pt}
\begin{itemize}
\setlength{\itemsep}{0pt}
\item[{\rm(i)}] There is a Lefschetz decomposition
\begin{align}
H^k_{d\dl}(M)& = \bigoplus_r\, L^r\, PH_{d\dl}^{k-2r}(M)~,\notag\\
\CH^k_{d\dl}(M)&=\bigoplus_r\, L^r \, P\CH_{d\dl}^{k-2r}(M)~.\notag
\end{align} 
\item[{\rm(ii)}] Lefschetz property: the Lefschetz operator defines an isomorphism 
\bes
L^{n-k}:\ \ H_{d\dl}^k(M) \cong H_{d\dl}^{2n-k}(M) \qquad {\rm for~} k \leq n~.
\ees
\end{itemize}
\label{ddlLef}
\end{thm}

\subsubsection{Duality with $d+\dl$ cohomology}\label{ddldual}

As is evident, $H^*_{d\dl}$ and $H^*_{d+\dl}$ share many properties.  Indeed, just by comparing the expressions of their respective Laplacians in \eqref{cohah} and \eqref{cohbh}, one finds
\begin{lem} \label{Laprel}
The Laplacians of the $d+\dl$ and $d\dl$ cohomology satisfy 
\bes
* \,\Delta_{d+\dl} = \Delta_{d\dl} *~. 
\ees  
\end{lem}
The following proposition then follows straightforwardly. 
\begin{prop}On a symplectic manifold $(M,\om)$, $\CH^{k}_{d+\dl}\,, \CH^{2n-k}_{d+\dl}\, ,\CH^k_{d\dl}\, , \CH^{2n-k}_{d\dl}$ are all isomorphic.  For $k\leq n$, we have the diagram
\bes
\xymatrix@R=40pt@C=60pt{
\CH^k_{d+\dl}(M) \ar
@{<->}
[r]^{*}
\ar @<-1ex>[d]_{L^{n-k}}  
&~\CH^{2n-k}_{d\dl}(M) 
\ar@<-1ex>[d]_{{\La}^{n-k}}
\\
\CH^{2n-k}_{d+\dl}(M) 
\ar @<-1ex>[u]_{{\La}^{n-k}} \ar@{<->}[r]_{*} 
&~ \CH^{k}_{d\dl}(M)
\ar@<-1ex>[u]_{L^{n-k}} 
  }
\ees
\end{prop}
The uniqueness of the harmonic representative in each cohomology class then imply the following.
\begin{cor}
On a compact symplectic manifold $(M,\om)$, $H^k_{d+\dl}(M)\cong H^{2n-k}_{d\dl}(M)$ and hence $\dim H^k_{d+\dl}(M)=\dim H^{2n-k}_{d\dl}(M)\,$.
\end{cor}
Being isomorphic, the cohomologies, $H^*_{d+ \dl}(M)$ and $H^*_{d\dl}(M)\,,$ are also naturally paired.
\begin{prop}On a compact symplectic manifold $(M,\om)$, the natural pairing
\bes
H^k_{d+\dl}(M) \otimes H^{2n-k}_{d\dl}(M) \longrightarrow  \mathbb{R}~,
\ees
defined by 
\bes
A \otimes A' \longrightarrow \int_M A \w A'~,
\ees
is non-degenerate, that is, a perfect pairing.
\end{prop}
\begin{proof} Notice first that the integral is well defined; that is, it is independent of the choice of the representative in either of the two cohomology classes.  To show non-degeneracy, we can then choose $A$ and $A'$ to be the respective harmonic representatives.  In particular, let $A\in \CH^k_{d+\dl}(M)$ and $A'=*A \in \CH^{2n-k}_{d\dl}(M)$.  We thus have for $A\neq 0\,$,
\bes
A\otimes * A \longrightarrow \int_M A \w * A = \| A\|^2  > 0~.
\ees
\end{proof}

\subsection{Example: Kodaira-Thurston nilmanifold}\label{KTsec}

It is helpful to have an explicit example showing clearly the differences between the different cohomologies discussed above.  For this we consider the following Kodaira-Thurston nilmanifold.    

Let $M=M^4$ be the nilmanifold defined by taking $\mathbb{R}^4$ and modding out by the identification $$(\xa, \xb, \xc, \xd) \sim (\xa + a, \xb + b, \xc + c, \xd + d - b \,\xc)~,$$  with $a, b, c, d \in \mathbb{Z}\,$.   The resulting manifold is a torus bundle over a torus (or more specifically here, an $S^1$ bundle over $T^3$) with a basis of cotangent one-forms given by
\be\label{cotvect}
\ea = d\xa\, , \qquad \eb=d\xb\, , \qquad \ec = d\xc\, , \qquad \ed = d\xd + \xb\, d\xc\, .
\ee
We take the symplectic form to be
\be\label{exsymp}
\om = \ea \w \eb + \ec \w \ed~.
\ee
Such a symplectic nilmanifold, discussed by Kodaira \cite{Kodaira} and Thurston \cite{Thurston}, admits a complex structure, though not a K\"ahler stucture since its first Betti number $b_1=3$ is odd.  A compatible almost complex structure can be expressed in terms of a decomposable $(2,0)$-form \footnote{We use the notation $e_{i_1i_2...i_k}= e_{i_1}\w e_{i_2} \w \ldots \w e_{i_k}\,$.}
\be\label{exOm}
\Omega=(\ea + i\, \eb) \w (\ec + i\, \ed) = \left(e_{13} - e_{24}\right) + i \left(e_{23} + e_{14}\right)~.
\ee
However, $\Omega$ is not closed;
\bes\label{exOmd}
d\, {\rm Re}\, \Omega =0 ~ , \qquad  d\, {\rm Im}\, \Omega = - e_{123}~.
\ees
Hence, the almost complex structure is not integrable.

The various symplectic cohomologies can be calculated by writing out explicitly the global differential forms.   The globally defined forms will not depend on the fiber $x_4$ coordinate.    A basis for the various cohomologies of Kodiara-Thurston manifold are given in Table \ref{extable}.

\begin{table}[!h]
{\renewcommand{\arraystretch}{1.25} 
\begin{tabular}{c |c | l | l | l | c }
Cohomology  & $k=0$ &~~$k=1$ & ~~~~~~~~~ $k=2$ & ~~~~~~~~~$k=3$ & $k=4$ \\
\hline
$H^k_d$ &  $1$ & $\ea,\eb,\ec$&$ \om, e_{12}-e_{34}, e_{13},e_{24} $&  $\om \w \ea, \om \w \eb, \om \w \ed$&$ \frac{1}{2}\,\om^2 $ \\
$H^k_{\dl}$ & $1$ & $\ea, \eb,\ed $ &$  \om, e_{12}-e_{34}, e_{13},e_{24}$& $\om \w \ea, \om \w \eb, \om \w \ec $& $ \frac{1}{2}\,\om^2 $\\
$H^k_{d+\dl}$ & $1$ & $ \ea,\eb,\ec$  &$  \om, e_{12}-e_{34}, e_{13},e_{24}, e_{23} $& $\om \w \ea, \om \w \eb, \om \w \ec $&$\frac{1}{2}\,\om^2 $\\
$H^k_{d\dl}$ & $1$ & $\ea, \eb,\ed $  & $  \om, e_{12}-e_{34}, e_{13},e_{24}, e_{14} $ & $\om \w \ea, \om \w \eb, \om \w \ed $&$\frac{1}{2}\,\om^2 $
\end{tabular}}

\ 
 \caption{Bases for $H_d$, $H_{\dl}$, $H_{d+\dl}$ and $H_{d\dl}$ 
of the Kodaira-Thurston fourfold in terms of the one-forms $e_i$ \eqref{cotvect} and symplectic form $\om$ \eqref{exsymp}.\label{extable}}
\

 \end{table}
 
Notice first that for $k$ even, all five cohomologies share at least one element.  This follows from the general fact that powers of the symplectic form, $\om^m$ for $m=0,1,\ldots,n\,$, are always $d$- and $\dl$-closed, and hence, they are non-trivial elements for the five cohomologies.  
 
Let us also point out certain properties of some of the differential $k$-forms on $M\,$.  For $k=1$, $e_3$ is $d$- and $\dl$-closed.  However, it is also $\dl$-exact as $e_3 =\dl e_{14}$.  The form, $e_4$, on the other hand, is $\dl$-closed, but it is not $d$-closed.  For $k=2$, $e_{23}$ is certainly $d$-closed; however, it is also $d$-exact and $\dl$-exact, $e_{23}=de_4 = \dl e_{124}$, but not $d\dl$-exact.  Thus $e_{23}$ is an explicit example showing that the $d\dl$-lemma fails for $M^4$.  (The dual submanifold associated with $e_{23}$ has an interesting property that we discuss in Section 4, Example \ref{currentex}.)  Also noteworthy is $e_{14}$, which is $d\dl$-closed but not closed under either $d$ or $\dl$.
 
With the Kodaira-Thurston nilmanifold not satisfying the $d\dl$-lemma, we also see that the strong Lefschetz property does not hold for the de Rham and $\dl$-cohomology.  However, as required, strong Lefschetz certainly does hold for $H^*_{d+\dl}$ and $H^*_{d\dl}$. 
From Table \ref{extable}, we also see the natural pairing between $H^k_{d+\dl}$ and $H^{4-k}_{d\dl}$ while $H^*_d$ and $H^*_{\dl}$ pair with themselves.  

\

 
\section{Dual currents of submanifolds and primitive cohomology}

A striking feature of the new symplectic cohomologies introduced in the last section is that they all commute with Lefschetz decomposition and hence naturally led us to cohomologies on the space of primitive forms.  One of the two primitive cohomologies, $PH_{\dpdl}(M)$,
 consists of primitive elements that are $d$-closed.  Hence, elements of this primitive cohomology are also elements of the de Rham cohomology.  With de Rham's theorem relating $H_d(M)$ to the homology of singular chains, a natural question is what special subsets of cycles/chains of a smooth compact symplectic manifold $M$ are dual to elements of primitive cohomologies.  In this section, we begin to explore this issue by analyzing the dual currents of the special submanifolds (e.g., coisotropic, isotropic, and symplectic) on a symplectic manifold $M$.\footnote{Recall that a submanifold $X\subset M$ is coisotropic if for every $x\in X$, the symplectic complement of the  vector subspace $T_xX\subset T_xM$ is also in $T_xX\,$, i.e. $(T_xX)^\om \subset T_x X\,$.  Moreover, $X$ is lagrangian if $(T_xX)^\om = T_x X\,$, isotopic if $T_xX \subset (T_xX)^\om$, and symplectic if $T_xX \cap (T_xX)^\om = \{0\}$.}  We will find that the dual currents of lagrangians and coisotropic submanifolds are in fact primitive, and thus, they can be considered as possible dual elements of primitive cohomologies.  We then introduce a homology on the subset of chains that are coisotropic.

Let $X\subset M$ be a compact codimension $m$ submanifold, possibly with boundary.  The dual current associated with $X$ denoted by $\rx$ is defined by
\begin{equation*}
\int_X i^* \a = \int_M \a \w \rx~,
\end{equation*}
where $i:X\hookrightarrow M$ is the embedding map and $\a$ is an arbitrary test $(2n-m)$-form \cite{deRham}.  If $\a$ is taken to be an element of the de Rham cohomology class, then the dual current is just the standard Poincar\'e dual, or equivalently, the Thom current of the normal bundle. 

For the special submanifolds of interest on $M$, the dual current can be expressed simply in local coordinates (see, for example, \cite{McDuff}).  In a local tubular neighborhood $U$ of $X\subset M$ (assumed here not to contain any boundary), we can work in the local Darboux coordinates $(p_1,\ldots,p_n,q_1,\ldots,q_n)$ such that the symplectic form takes the standard form 
\bes\label{sympf}
\om= \sum dp_i \w dq_i~,
\ees
and $X$, having codimension $m=m_1 + m_2\,$, is the zero locus of 
\bes\label{zlocus}
p_1=\ldots=p_{m_1}=q_1=q_2=\ldots=q_{m_2}=0~.
\ees
The dual current then has the canonical form 
\bes  
\rx = \delta(p_1,\ldots,p_{m_1},q_1,\ldots,q_{m_2}) \,dp_1\w \ldots \w dp_{m_1}\w dq_1 \w \ldots \w dq_{m_2}~,
\ees
with the $\delta$-function distribution defined as $\delta(f)=f(0)$.  Clearly, for any closed submanifold, we find $d \rx =0\,$.

Bahramgiri \cite{Bahramgiri} has shown that the dual current of a closed submanifold is not only $d$-closed but also $\dl$-closed if and only if $X$ is coisotropic.   A sharper statement is that the $\dl$-closedness property is due to the fact that $\rx$ is primitive when $X$ is coisotropic.  We give the property of the dual current for coisotropic, isotropic, and symplectic submanifolds.
  
\begin{lem}Let $X\subset M$ be an embedded compact submanifold with dual current $\rx$.  Then: 
\begin{itemize}
\setlength{\parsep}{0pt}
\setlength{\itemsep}{0pt}
\item[{\rm(i)}] $\rx$ is primitive if and only if $X$ is coisotropic.
\item[{\rm(ii)}] $\ss\,\rx$ is primitive if and only if $X$ is isotropic.
\item[{\rm(iii)}] If $X$ is symplectic with codimension $m=2l$, then $\La^k\rx \neq 0$ for $k = 1, \ldots, l\,\,$.
\end{itemize}
\end{lem}

\begin{proof}One can prove the lemma by making use of the canonical local coordinates described above and applying the defining property of the submanifolds.  The proof we give here follows from the integral definition of the dual current.

(i) Note first that if the submanifold is codimension $m=1$, the statement holds trivially, as $X$ would be automatically coisotropic and the dual 1-current is trivially primitive.  Now in general, a coisotropic submanifold $X^{2n-m}\subset M$ of codimension $m$ satisfies the condition
\bes
\om^{n-m+1}\Big|_{X^{2n-m}} = 0~.
\ees
In the case $m=n$, this condition reduces to the lagrangian condition.  Integrating a test form $\a= \om^{n-m+1}\w \b_{m-2}$ with $\b_{m-2}$ arbitrary over $X^{2n-m}$, we find
\bes
0= \int_{X^{2n-m}} i^*(\om^{n-m+1}\w \b_{m-2}) = \int_M \om^{n-m+1}\w \b_{m-2} \w \rho_X~.
\ees
But since $\b_{m-2}$ is arbitrary, we must have $\om^{n-m+1} \w \rho_X =0\,$, which is precisely the primitive condition (Defintion \ref{pdef}) for the $m$-current $\rho_X$.  Thus, $X$ coisotropic implies $\rx$ primitive.  Conversely, if $X$ is not coisotropic, then there exists a test form $\b_{m-2}$ such that the integral is nonzero.  Then $\rx$ cannot be primitive, as $\om^{n-m+1} \w \rho_X \neq 0\,$.

(ii) Similarly, let $X^k$ be an isotropic submanifold with $1<k\leq n$.  The cases $k=0,1$ are trivially true.  The isotropic condition is that 
\bes 
\om \,\big|_{X^k} =0~.
\ees
Now let the test form $\a= \om \w \b_{k-2}$ with $\b_{k-2}$ arbitrary.  We find
\bes 
0=\int_{X^k} i^*(\om \w \a_{k-2}) = \int_M \om \w \a_{k-2} \w \rho_{X}~,
\ees
which implies $L \,\rho_{X} =0$ or, equivalently, $\La (*_s\, \rho_{X}) =0$.  Therefore,  the $k$-current $*_s \, \rho_{X}$ is primitive or $\rho_{X}= \frac{1}{(n-k)!}\, L^{n-k} \sigma_X$ with $\sigma_X$ being the primitive $k$-current.  The converse statement similar to the proof (i) is straightforward.

(iii) Note that $\La^l \rho_X \neq 0$ implies $\La^k \rho_X\neq 0$ for $k=1,2,\ldots, l-1$.  So we only need to show $\La^l \rho_X \neq 0$ for $X$ a symplectic submanifold.  Such a statement basically follows from an argument due to Bahramgiri \cite{Bahramgiri}, which we reproduce here.    
For a symplectic submanifold $X^{2(n-l)}$, we have
\begin{align*}
0\neq \int_{X^{2(n-l)}} i^* \left(\dfrac{\om^{n-l}}{(n-l)!}\right)&= \int_M \,\dfrac{\om^{n-l}}{(n-l)!} \w \rho_{X} =  \int_M  \left(\ss \dfrac{\om^l}{l!}\right) \w \rho_{X} \\
& =  \int_M \frac{\om^l}{l!} \w \ss \rho_{X} =  \int_M \ss \frac{1}{l!}\,\La^l \, \rho_{X}.
\end{align*}
Hence, we find $\La^l \rho_{X} \neq 0\,$.
\end{proof}

Note that the above lemma holds without regard to whether the submanifold is closed or not.  If $X$ is in fact closed, then as mentioned it is clear that $\rx$ is $d$-closed.

\begin{ex}\label{currentex}
We give some examples of dual currents of closed symplectic and lagrangian submanifolds of the Kodaira-Thurston manifold $M^4$, discussed in Section \ref{KTsec}.  With the symplectic form  $\om = e_{12} + e_{34}$,  submanifolds that wrap around $(x_1, x_2)$ and $(x_3, x_4)$ are symplectic. Their dual currents $\rho_S$ can, for instance, take the form
\bes
(x_1,x_2): \quad \rho_S=\delta(x_3,x_4)\, e_{34}~,\qquad\qquad  (x_3,x_4):\quad \rho_S=\delta(x_1,x_2)\, e_{12}~.
\ees
Clearly $\La \rho_S \neq 0\,$.  Submanifolds that wrap around $(x_1, x_3), (x_2, x_4)$, or $(x_1, x_4)$ are lagrangian.  Note that there is no submanifold wrapping around $(x_2, x_3)$ as the $S^1$ bundle has no zero section.  A representative set of dual currents $\rho_L$ for these lagrangians are  
\begin{align*}
(x_1,x_3):\quad& \rho_L=\delta(x_2,x_3)\, e_{24}~,\\
(x_2,x_4):\quad& \rho_L=\delta(x_1,x_3)\, e_{13}~,\\
(x_1,x_4):\quad& \rho_L=\delta(x_2,x_3)\, e_{23}~.
\end{align*}  
And certainly all $\rho_L$ are primitive.

It is interesting to point out that although $e_{23}$ represents a non-trivial class of $H^2_{d+\dl}(M^4)\,$, it is a trivial element in $H^2_d(M^4)$.  The  submanifold, $L_{14}$, wrapping around $(x_1,x_4)$ is a lagrangian but not strictly a two-cycle.  Further, using the compatible almost complex structure of \eqref{exOm}, we find that ${\rm Re}\,\Omega |_{L_{14}} = 0$ and ${\rm Im}\,\Omega |_{L_{14}} = {\rm vol}(L_{14})$ and so that $L_{14}$ is in fact a special lagrangian submanifold.  However, since $d\, {\rm Im}\,\Omega \neq 0$, $L_{14}$ is actually an example of a generalized calibrated submanifold as discussed in \cite{GutP}      
\end{ex}

\subsection{Homology of coisotropic chains}\label{homology}

Having seen that closed coisotropic submanifold are associated with $d$-closed primitive dual currents, we proceed now to describe how a primitive cohomology can be dual to a homology on coisotropic chains.  

Beginning with the primitive cohomology 
\bes 
PH^k_{\dpdl}(M) = \dfrac{\ker d \cap \CB^k(M)}{d\dl \CB^k(M)}~ ,
\ees
we note that while the exterior derivative $d$ is dual to the boundary operation $\partial$ acting on submanifold, the dual chain operation for the $\dl$ operator is not as clear.  We will sidestep this issue here by introducing the following $d$-primitive cohomology
\begin{equation}\label{ddprim}
PH^k_d(M) = \dfrac{\ker d \cap \CB^k(M)}{d\, \CB'^{k-1}(M)} = \dfrac{\ker d \cap \CB'^k(M)}{d\, \CB'^{k-1}(M)}~,
\end{equation}
where $\CB'^{k}(M)\subset \CB^k(M)$ is the space of $\dl$-closed primitive forms; that is, $B\in \CB'(M)$ if $\La B=0$ and $\dl B =0$.  Such a primitive space has the following desirable property.
\begin{lem} Let $B_k\in \CB^k(M)$.  Then $dB_k$ is primitive if and only if $\dl B_k =0$; that is, $B_k \in \CB'^k(M)$.  In particular, for $k <n, \ d:  \CB'^k(M) \to \CB'^{k+1}(M)$.   \label{dlplem}
\end{lem}
\begin{proof}This follows simply from Lemma \ref{diffB}.  If $dB_k$ is primitive, then $0 =-\La dB_k = \dl B_k =0$. Conversely,  assume now $B_k\in \CB'^k(M)$; then $\dl B_k = - H B^1_{k-1}=0$.   Therefore, $dB_k = B^0_{k+1} + L B^1_{k-1}= B^0_{k+1}\,$, which is primitive.  This shows that $d:  \CB'^k(M) \to \CB^{k+1}(M)$.  Continuing, we have $\dl B^0_{k+1} = \dl d B_k = - d \dl B_k =0\,$, which implies in particular $B^0_{k+1}\in \CB'^{k+1}(M)$.    
\end{proof}

With $\dl: \CB^k \to \CB^{k-1}\,$ and ${\dl}^2=0$, it is clear that the image of $d\dl \CB^k$ is contained within $d\,\CB'^{k-1}\,$.  The two images coincide if the $d\dl$-lemma holds.  Therefore, we see that $PH^k_d(M) \subseteq PH^k_{\dpdl}(M)$, and this implies in particular that $PH^k_d(M)$ must also be finite dimensional.
\begin{prop}
On a compact symplectic manifold $(M,\om)$, for $k\leq n$, 
$$\dim PH^k_d(M) \leq  \dim PH^k_{\dpdl}(M) < \infty\,.$$
\end{prop}

\

Let us now consider the cohomology $PH^k_d(M)$ defined on the space of primitive currents instead of forms.  This leads us to a natural dual homology on $(M,\om)\,$.  

Let $\mC_l(M)$ be the space of $l$-chains that are coisotropic.  If the coisotropic chains contain boundaries, we require that their boundaries are also coisotropic.  With $\partial$ denoting the operation of taking the oriented boundary, we introduce a homology on coisotropic chains
\be\label{phol}
PH_{l}(M) = \dfrac{\ker \partial \cap \mC_l(M)}{\partial\, \mC_{l+1}(M)}~,
\ee
for $ n \leq l < 2n\,$.  

The homology $PH_{l}(M)$ can be seen to be naturally dual to $PH^{2n-l}_d(M)$ of \eqref{ddprim}.  The requirement that any boundary of $\mC_{l+1}$ is also coisotropic ensures that $\partial: \mC_{l+1} \to \mC_{l}\,.$  This is precisely dual to the requirement that $\dl \CB' =0$, which ensures $d: \CB'^{2n-l-1} \to \CB'^{2n-l}$.  Explicitly, let $\Ca, \Cb, \Cc \in \mC_*(M)$ and suppose $\Cb = \Ca + \partial \Cc$; that is, the boundaries of $\Cc$ are coisotropic.   Integrating over the test form $\a$, we have
\begin{align*}
0&=\int_{\Cb} i^* \a \  - \int_{\Ca}  i^* \a \  - \int_{\partial \Cc} i^* \a \\
&=\int_M  \a \w \rb \  - \int_M \a \w \ra \  - \int_M d \a \w \rc \\
&= \int_M \a \w \left( \rb - \ra - (-1)^l d\rc\right)~,
\end{align*}
where $\ra, \rb, \rc \in \CB^*(M)$ are the respective dual primitive currents.  Since $\a$ is  arbitrary, we find $\rb- \ra = (-1)^l d\rc\,$; hence, $d\rc$ must also be primitive, and by Lemma \ref{dlplem}, $\dl \rc =0$ and therefore $\rc\in \CB'^{2n-l-1}(M)$.  The converse statement can also be shown straightforwardly, and we see that the dual current is an element of $\CB'(M)$ if and only if the boundary of the coisotropic chain is also coisotropic.

Furthermore, if we consider an infinitesimal symplectomorphism, not necessarily hamiltonian, the homology class of both the closed primitive current and the coisotropic cycle (i.e. a closed coisotropic chain) remains invariant.  For the closed primitive current, since the Lie derivative $\CL_V B=d(i_V B)$ must be primitive, Lemma \ref{dlplem} implies $i_V B \in \CB'(M)$.  For a closed coisotropic cycle of dimension $l$, an infinitesimal symplectomorphism sweeps out a coisotropic chain $C_{l+1}$, of one higher dimension.  That this $C_{l+1}$ is coisotropic can be shown by using local Darboux coordinates in the tubular neighborhood. 

We have shown that $PH^k_d(M)$ is finite dimensional over differential forms by comparison with $PH^k_{\dpdl}(M)$.  Demonstrating the same on the space of currents involves a more direct proof of finite dimensionality of $PH^k_d(M)$ over smooth forms.  This and other properties of primitive cohomologies will be discussed in follow-up papers \cite{TY, TY1}.

\

\section{Discussion}

\def\tM{{N}}

We have introduced new finite-dimensional symplectic cohomologies - $H^*_{d+\dl}(M)$ and $H^*_{d\dl}(M)\,$
- using the differential operators $d+\dl$ and $d\dl$.
\footnote{A third finite-dimensional cohomology $H^k_{\dadl}(M) = \frac{\ker (d+\dl) \cap \CA^k(M)}{({\rm im} \ d + {\rm im} \ \dl) \cap \CA^k(M)}$ originally introduced in the arXiv preprint version of this manuscript will be discussed elsewhere.}  These cohomologies are all isomorphic to the de Rham cohomology when the $d\dl$-lemma holds on $M$.   And conversely, when they differ from the de Rham cohomology, this implies the $d\dl$-lemma and the strong Lefschetz property both fail.  As K\"ahler manifolds satisfy the $d\dl$-lemma, the new cohomologies are particularly suited for distinguishing the more intricate geometries of non-K\"ahler symplectic manifolds.

It is interesting to compare the symplectic cohomologies with the known differential cohomologies on complex manifolds, $\tM$, not necessarily K\"ahler.   Besides the standard ones of deRham and Dolbeault, there are two others that have also been studied: the Bott-Chern cohomology 
\be\label{BC}
H^{p,q}_{BC}(\tM)= \frac{(\ker \pa \cap \ker \bpa) \cap \CA^{p,q}(\tM)}{ \im \pa\bpa \cap \CA^{p-1,q-1}(\tM)}~
\ee
and the Aeppli cohomology \cite{Aeppli}
\be\label{Aeppli}
H^{p,q}_A(\tM) = \frac{\ker \pa\bpa \cap \CA^{p,q}(\tM)}{(\im \pa + \im \bpa)\cap \CA^{p,q}(\tM)}~.
\ee
These two cohomologies are similarly paired and share many analogous properties that we have shown for the pair $H^k_{d+\dl}(M)$ and $H^k_{d\dl}(M)$ defined on symplectic manifold.  Indeed, both can be shown to be finite dimensional by constructing self-adjoint fourth-order differential operators \cite{KS, Bigolin} (see also \cite{Sch}).  Explicitly, the fourth-order operators are  
\be\label{BCO}
D_{BC} = (\pa\bpa)(\bpas\pas) +  (\bpas\pas)(\pa\bpa) + (\pas\bpa)(\bpas\pa) + (\bpas\pa)(\pas\bpa) + \la(\pas\pa + \bpas\bpa)~,
\ee
\be\label{AeppliO}
D_A = (\pa\bpa)^*(\pa\bpa) + (\pa\bpa)(\pa\bpa)^* +  (\pa\bpas)(\bpa\pas)  + (\bpa\pas)(\pa\bpas) + \la( \pa\pas +  \bpa\bpas)~,
\ee 
for real constants $\la>0$.  Indeed, they are analogous to $D_{d+\dl}$ and $D_{d\dl}$ in \eqref{saO} and \eqref{sbO}, respectively.

In the spirit of mirror symmetry - in the most general sense that certain properties of complex geometry are directly associated with certain properties of symplectic geometry - we should identify the symplectic differential pair $(d, \dl)$ with the complex pair $(\pa, \bpa)$.  An immediate question that arises is what is then the dual of the de Rham cohomology ($d=\pa +\bpa$) on a symplectic manifold $M$?  This suggests looking at the cohomology of type
\bes
H(M)=\frac{\ker (d+ \dl)}{\im (d+\dl)}~.
\ees
But notice that $(d+\dl)^2=0$ only if we consider the space of differential forms $\CA^*(M)$ partitioned not into fixed degrees, that is, $\CA^k(M)$, but into the space of even and odd degrees - $\CA^{ev}(M)$  and $\CA^{odd}(M)$.  This would then be a cohomology acting on the formal sums of even or odd differential forms.  Doing so, one can then show by means of a basis transformation (as in \cite[p.\,89]{Caval}) that $H(M)$ above is indeed isomorphic to the de Rham cohomology.

Having noted the similarities of cohomologies in complex and symplectic geometries, it is natural to consider extending the cohomologies to the generalized complex geometries introduced by Hitchin \cite{Hitchin} (see also \cite{Gualtieri, Caval}).  As generalized complex geometry brings together both complex and symplectic structures within one framework, there should certainly be an extension of the new symplectic cohomologies in the generalized complex setting, and it would be interesting to work them out explicitly (see \cite{TY3}).

\

\


\vskip 1cm
Department of Mathematics, University of California, Irvine, CA 92697 USA\\
{\it Email address:}~{\tt lstseng@math.uci.edu}

\vskip .5 cm
\noindent 
{Department of Mathematics, Harvard University\\
Cambridge, MA 02138 USA}\\
{\it Email address:}~{\tt yau@math.harvard.edu}

\end{document}